\documentclass{article}

\usepackage{cite}
\usepackage{arxiv}

\usepackage[utf8]{inputenc} 
\usepackage[T1]{fontenc}    
\usepackage{hyperref}       
\usepackage{url}            
\usepackage{booktabs}       
\usepackage{amsfonts}       
\usepackage{nicefrac}       
\usepackage{microtype}      
\usepackage{lipsum}		
\usepackage{graphicx}
\usepackage{doi}

\usepackage{color}

\usepackage{enumerate}
\usepackage{amssymb}
\usepackage{amsmath}
\usepackage{amsthm}
\usepackage{enumitem}
\usepackage{mathrsfs}
\usepackage{rotating}
\usepackage{bm}
\usepackage{xspace}

\newcommand{\R}{\mathbb{R}}

\newcommand{\var}{\mathbb{\varepsilon}}

\renewcommand{\rho}{\varepsilonrho}
\renewcommand{\phi}{\varepsilonphi}

\newtheorem{thm}{Theorem}[section]

\title{Fast reaction limit for a Leslie-Gower model including preys, meso-predators and top-predators}

\author{  
L. Desvillettes$^1$\thanks{Corresponding author.}
\\ 
\texttt{desvillettes@imj-prg.fr} \\
\And
L. Fiorentino$^2$\thanks{Corresponding author.}\href{https://orcid.org/0000-0002-0154-6035}{\includegraphics[scale=0.06]{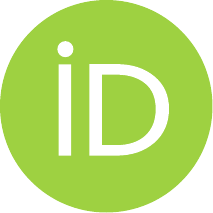}} \\ \texttt{ludovica.fiorentino@unina.it} \\ 
\And
T. Mautone$^3$ \\
\texttt{t.mautone@studenti.unina.it}  
\\ \\ $^1$Université Paris Diderot, Sorbonne Paris Cité \\ Institut de Mathématiques de Jussieu-Paris Rive Gauche, UMR 7586, CNRS, Sorbonne Universités, UPMC Univ. Paris 06\\ F-75013 \\ Paris \\ France \\
\\ \\ $^2$Dipartimento di Matematica e Applicazioni 'R.Caccioppoli' \\ Università degli Studi di Napoli Federico II \\ Via Cintia, Monte S.Angelo, 80126 Napoli \\ Italy \\
\\ \\ $^3$Gran Sasso Science Institute (GSSI)\\ Viale F. Crispi, 7, 67100 L'Aquila \\ Italy \\
}

\renewcommand{\shorttitle}{A reaction-diffusion model for Leslie-Gower model}

\begin{document}
\maketitle
\begin{abstract}
	We consider a system of three reaction-diffusion equations modeling the interaction between a prey species and two predators species including functional responses of Holly type-II and Leslie-Gower type. We propose a reaction-diffusion model with five equations with simpler functional responses which, in the fast reaction limit, allows to recover the zero-th order terms of the initially considered system. The diffusive part of the initial equations is however modified and cross diffusion terms pop up. 
    We first study the equilibria of this new system and show that no Turing instability appears. We then rigorously prove a partial result of convergence for the fast reaction limit (in 1D and 2D)
\end{abstract}


\section{Introduction}\label{intro}

\subsection{Presentation}
In predator-prey models, the use of complex functional responses is particularly efficient 
\cite{Holling1966, Ivlev1961, Beddington1975, DeAngelis1975, AG2000}. For example, the Holling-type II functional response \cite{Holling1966} models predators' behavior, namely describes the situation in which the predators will catch a limited number of available preys when they're abundant. On the other hand, the Beddington-DeAngelis functional response \cite{Beddington1975, DeAngelis1975} takes into account the competition between predators during the hunting. 

In the past, several works were written in order to obtain a formalization of the Holling-type II and Beddington-DeAngelis functional responses out of simple and realistic “microscopic models”. Metz and Diekmann \cite{MeDic1986}, and Geritz and Gyllenberg \cite{GeGyll2012} designed this kind of models for the Holling-type II functional response, Huisman and De Boer \cite{HuiBoe1997}, for the Beddington-DeAngelis one, and references therein. Metz and Diekmann proposed a system of three ODEs, in which the predators are divided in two classes, searching and handling, respectively. In the limit when the switch between handling and searching predators is infinitely fast (fast reaction limit), the original 
(Holling type-II) functional response is recovered. It is interesting to notice that starting with standard diffusion terms in the microscopic model 
and imposing that the diffusion is slower for handling predators than for searching predators, one ends up in the macroscopic model
with cross diffusion terms  (cf. \cite{LD_soresina} and \cite{CDesS2018}). 
As a consequence, the study of stability (appearance 
or not of Turing instability) is different from the study of stability of a macroscopic model in which standard diffusion are imposed.

In the present paper, we start from a (macroscopic) model of Leslie-Gower type of three equations (one equation for preys, 
one equation for meso-predators and one equation for top-predators) introduced in \cite{CdLFLM2023}, and propose for this type of functional response 
(\cite{Les1, Les2, Les3}, \cite{CCX}, \cite{CC}, \cite{Yu2014} and references therein)
a microscopic model of five equations. 
\nocite{Les1, Les2, Les3}

We present here the main features of this model, starting with the $0$-th order terms (in terms of derivatives). The preys 
follow a logistic equation in which a prey-predator term (the predator in this term is the meso-predator) is added. This term includes the effect of the presence of top predators on the predation rate of meso-predators. 
Like in \cite{MeDic1986}, we have divided the meso and top-predator populations into two classes, searching and handling, respectively. Searching meso-predators become handling ones at rate $\tilde{\alpha}$ proportional to the predation rate of preys and come back to the searching class with a constant rate $\tilde{\gamma}$, analogously searching top-predators become handling ones at rate $\tilde{\beta}$, which is proportional to meso-predator population and come back to the searching class with a constant rate $\tilde{\eta}$. It is assumed that only handling meso-predators can contribute to the reproduction, in fact they give birth to a searching predator at constant rate $\Gamma$, while the meso-predators mortality rate $\tilde{\mu}$ (for deaths not due to the predation of the top-predators) is constant and equal for the two classes. Finally, the top predators follow a logistic equation where the carrying capacity is proportional to the number of meso-predators (Leslie-Gower functional response).
\par
Furthermore, it is assumed that the diffusion coefficient $d_2^{(1)}$ and $d_3^{(1)}$ of searching meso and top-predators are, respectively, greater then the diffusion coefficient $d_2^{(2)}$, $d_3^{(2)}$ of handling ones. As regard the searching-handling switch of meso-predators $\dfrac{1}{\delta}$, and the searching-handling switch of top-predators $\dfrac{1}{\varepsilon}$, it assumed that they are on a much faster time-scale than the reproduction and mortality processes. Under these assumptions, 
our system writes

\begin{equation}
\label{lim_sist}
        \begin{cases}
            \dfrac{\partial P^{\varepsilon, \delta}}{\partial \Tilde{t}} = \Tilde{r} \left(1 - \dfrac{P^{\varepsilon, \delta}}{K} \right) P^{\varepsilon, \delta} - \dfrac{\Tilde{\alpha} P^{\varepsilon, \delta} M_s^{\varepsilon, \delta}}{1 + \Tilde{c}\, \left(T_s^{\varepsilon, \delta}+T_h^{\varepsilon, \delta}\right)} + d_1 \Delta P^{\varepsilon, \delta}, \\[10pt]
            \dfrac{\partial M_s^{\varepsilon, \delta}}{\partial \Tilde{t}} = \dfrac{1}{\delta} \left[  - \dfrac{\Tilde{\alpha} P^{\varepsilon, \delta} M_s^{\varepsilon, \delta}}{1 + \Tilde{c}\, \left(T_s^{\varepsilon, \delta}+T_h^{\varepsilon, \delta}\right)} + \Tilde{\gamma} M_h^{\varepsilon, \delta} \right] + \Gamma M_h^{\varepsilon, \delta} - \Tilde{\mu} M_s^{\varepsilon, \delta} - \Tilde{\beta} M_s^{\varepsilon, \delta} \, T_s^{\varepsilon, \delta} + d_2^{(1)} \Delta M_s^{\varepsilon, \delta}, \\[10pt]
            \dfrac{\partial M_h^{\varepsilon, \delta}}{\partial \Tilde{t}} = \dfrac{1}{\delta} \left[  \dfrac{\Tilde{\alpha} P^{\varepsilon, \delta} M_s^{\varepsilon, \delta}}{1 + \Tilde{c}\, \left(T_s^{\varepsilon, \delta}+T_h^{\varepsilon, \delta}\right)} - \Tilde{\gamma} M_h^{\varepsilon, \delta} \right] - \Tilde{\mu} M_h^{\varepsilon, \delta} - \Tilde{\beta} M_h^{\varepsilon, \delta} \, T_s^{\varepsilon, \delta} + d_2^{(2)} \Delta M_h^{\varepsilon, \delta}, \\[10pt]
            \dfrac{\partial T_s^{\varepsilon, \delta}}{\partial \Tilde{t}} = \dfrac{1}{\varepsilon} \left[ - \Tilde{\beta} (M_s^{\varepsilon, \delta} + M_h^{\varepsilon, \delta}) T_s^{\varepsilon, \delta} + \Tilde{\eta} T_h^{\varepsilon, \delta} \right] + \Tilde{s} \,T_s^{\varepsilon, \delta}\,  \left(1 - \dfrac{T_h ^{\varepsilon, \delta} + T_s^{\varepsilon, \delta} }{\Tilde{m}(M_s^{\varepsilon, \delta} + M_h^{\varepsilon, \delta})} \right) + d_3^{(1)} \Delta T_s^{\varepsilon, \delta}, \\[10pt]
            \dfrac{\partial T_h^{\varepsilon, \delta}}{\partial \Tilde{t}} = \dfrac{1}{\varepsilon} \left[  \Tilde{\beta} (M_s^{\varepsilon, \delta} + M_h^{\varepsilon, \delta}) T_s^{\varepsilon, \delta} - \Tilde{\eta} T_h^{\varepsilon, \delta} \right] + \Tilde{s} \,T_h^{\varepsilon, \delta}\,  \left(1 - \dfrac{T_h ^{\varepsilon, \delta} + T_s^{\varepsilon, \delta} }{\Tilde{m}(M_s^{\varepsilon, \delta} + M_h^{\varepsilon, \delta})} \right) + d_3^{(2)} \Delta T_h^{\varepsilon, \delta}, \\
        \end{cases}
    \end{equation}
where $P^{\varepsilon, \delta}:= P^{\varepsilon, \delta}(\Tilde{t}, \Tilde{x})$ is the density of preys; while $M_s^{\varepsilon, \delta} := M_s^{\varepsilon, \delta}(\Tilde{t}, \Tilde{x})$, $T_s^{\varepsilon, \delta} := T_s^{\varepsilon, \delta}(\Tilde{t}, \Tilde{x})$, $M_h^{\varepsilon, \delta} := M_h^{\varepsilon, \delta}(\Tilde{t}, \Tilde{x})$ and $T_h^{\varepsilon, \delta} := T_h^{\varepsilon, \delta}(\Tilde{t}, \Tilde{x})$ are the respective densities of searching and handling meso and top-predators 
at time $\Tilde{t} \ge 0$ and point $\Tilde{x} \in \Omega \subset \R^N$. To system \eqref{lim_sist},
 we append the homogeneous Neumann (no-flux) boundary conditions

\begin{equation}
\label{BC}
    \nabla_x P^{\varepsilon, \delta}\cdot\textbf{n}=0, \quad \nabla_x M_s^{\varepsilon, \delta}\cdot\textbf{n}=0, \quad  \nabla_x M_h^{\varepsilon, \delta}\cdot\textbf{n}=0\quad \nabla_x T_s^{\varepsilon, \delta} \cdot \textbf{n}=0, \quad \nabla_x T_h^{\varepsilon, \delta} \cdot \textbf{n}=0, \quad \text{on} \ \partial\Omega\times\R^+, 
\end{equation}
and the ($\varepsilon, \delta$-independent) initial data
\begin{equation}
\label{soinit}
\begin{aligned}
    P^{\varepsilon, \delta}(0, \Tilde{x}) =& P_{in}( \Tilde{x}), \quad  M_s^{\varepsilon, \delta}(0, \Tilde{x}) = M_{s,in}( \Tilde{x}), \quad   M_h^{\varepsilon, \delta}(0, \Tilde{x}) = M_{h,in}( \Tilde{x}), \quad \\
    &T_s^{\varepsilon, \delta}(0, \Tilde{x}) =  T_{s,in}( \Tilde{x}), \quad   T_h^{\varepsilon, \delta}(0, \Tilde{x}) = T_{h,in}( \Tilde{x}) ,  \quad \text{on} \ \Omega .
\end{aligned}
\end{equation}
It is assumed that all the biological parameters of the system \eqref{lim_sist} are positive, and that all initial data are nonnegative.
\medskip

 The formal limit when $\var \to 0$ (with $\delta>0$ fixed) of 
this sytem writes 
\begin{equation}
\label{lim_sist_del}
        \begin{cases}
            \dfrac{\partial P^{\delta}}{\partial \Tilde{t}} = \Tilde{r} \left(1 - \dfrac{P^{\delta}}{K} \right) P^{ \delta} - \dfrac{\Tilde{\alpha} \Tilde{\gamma} P^{ \delta} (M_s^{ \delta} +M_h^{\delta}) }{\Tilde{\gamma} +\Tilde{\alpha} P^{\delta} + \Tilde{c}\, \Tilde{\gamma} T^{\delta}} + d_1 \Delta P^{\delta}, \\[10pt]
            \dfrac{\partial M_s^{ \delta}}{\partial \Tilde{t}} = \dfrac{1}{\delta} \left[  - \dfrac{\Tilde{\alpha} P^{\delta} M_s^{ \delta}}{1 + \Tilde{c}\, T^{ \delta}} + \Tilde{\gamma} M_h^{ \delta} \right] + \Gamma M_h^{ \delta} - \Tilde{\mu} M_s^{\delta} - \dfrac{ \Tilde{\beta} \Tilde{\eta} M_s^{\delta} \, T^{\delta} }{\tilde{\eta} + \tilde{\beta}  (M_s^{ \delta} +M_h^{\delta})    } + d_2^{(1)} \Delta M_s^{\delta}, \\[10pt]
            \dfrac{\partial M_h^{ \delta}}{\partial \Tilde{t}} = \dfrac{1}{\delta} \left[  \dfrac{\Tilde{\alpha} P^{\delta} M_s^{ \delta}}{1 + \Tilde{c}\, T^{ \delta}} - \Tilde{\gamma} M_h^{ \delta} \right] - \Tilde{\mu} M_h^{ \delta}  -  \dfrac{ \Tilde{\beta} \Tilde{\eta} M_h^{\delta} \, T^{\delta} }{\tilde{\eta} + \tilde{\beta}  (M_s^{ \delta} +M_h^{\delta})    } + d_2^{(2)} \Delta M_h^{ \delta}, \\[10pt]
            \dfrac{\partial T^{ \delta}}{\partial \Tilde{t}} =  \Tilde{s} \,T^{ \delta}\,  \left(1 - \dfrac{T ^{\delta} }{\Tilde{m}(M_s^{\delta} + M_h^{\delta})} \right) + \Delta \left( \dfrac{  d_3^{(1)}\tilde{\eta} + d_3^{(2)} \tilde{\beta}  (M_s^{ \delta} +M_h^{\delta})     }{ \tilde{\eta} + \tilde{\beta}  (M_s^{ \delta} +M_h^{\delta})   }\,  T^{\delta} \right), \\
        \end{cases}
    \end{equation}
with homogeneous Neumann (no-flux) boundary conditions

\begin{equation}
\label{BCdel}
    \nabla_x P^{ \delta}\cdot\textbf{n}=0, \quad \nabla_x M_s^{ \delta}\cdot\textbf{n}=0, \quad  \nabla_x M_h^{ \delta}\cdot\textbf{n}=0\quad \nabla_x T^{ \delta} \cdot \textbf{n}=0, \quad \text{on} \ \partial\Omega\times\R^+, 
\end{equation}
and the initial data
\begin{equation}
\label{soinitdel}
    P^{\delta}(0, \Tilde{x}) = P_{in}(\Tilde{x}), \quad  M_s^{\delta}(0,\Tilde{x}) = M_{s,in}(\Tilde{x}), \quad   M_h^{\delta}(0,\Tilde{x}) = M_{h,in}(\Tilde{x})\quad  T^{ \delta}(0,\Tilde{x}) =  T_{s,in}(\Tilde{x}) + T_{h,in}(\Tilde{x}) ,  \quad \text{on} \ \Omega .
\end{equation}

Starting then from system \eqref{lim_sist_del} -- \eqref{soinitdel} and letting $\delta \to 0$, we end up at the formal level with the system
%
\begin{equation}\label{cross_sist}
        \begin{cases}
            \dfrac{\partial P}{\partial \Tilde{t}} = \Tilde{r} \left(1 - \dfrac{P}{K} \right) P - \dfrac{\Tilde{\alpha} \Tilde{\gamma} PM}{\Tilde{\gamma} + \Tilde{\alpha} P + \Tilde{c} \Tilde{\gamma} T} + d_1 \Delta P, \\[10pt]
            \dfrac{\partial M}{\partial \Tilde{t}} = \dfrac{\Gamma \Tilde{\alpha} PM}{\Tilde{\gamma} + \Tilde{\alpha} P + \Tilde{c} \Tilde{\gamma} T} - \Tilde{\mu} M - { \dfrac{\Tilde{\eta} \Tilde{\beta} MT}{\Tilde{\eta} + \Tilde{\beta}M}}  + \Delta \left( \dfrac{d_2^{(1)} \Tilde{\gamma} (1+ \Tilde{c} T) + d_2^{(2)} \Tilde{\alpha}P}{\Tilde{\gamma} + \Tilde{\alpha} P + \Tilde{c} \Tilde{\gamma} T} M \right),  \\[10pt] 
            \dfrac{\partial T}{\partial \Tilde{t}} = \Tilde{s} \left(1 - \dfrac{T}{\Tilde{m} M} \right) T  + {\Delta \left( \dfrac{d_3^{(1)} \Tilde{\eta} + d_3^{(2)} \Tilde{\beta}M}{\Tilde{\eta} + \Tilde{\beta} M} T \right)},
        \end{cases}
    \end{equation}
with homogeneous Neumann (no-flux) boundary conditions

\begin{equation}
    \label{BC_lim}
     \nabla_x P\cdot\textbf{n}=0, \quad \nabla_x M\cdot\textbf{n}=0,
     \quad \nabla_x T \cdot \textbf{n}=0,  \quad \text{on} \ \partial\Omega\times\R^+,
\end{equation}
  and the initial data
\begin{equation}
\label{soinitdel2}
    P(0,\Tilde{x}) = P_{in}(\Tilde{x}), \quad  M(0,\Tilde{x}) = M_{s,in}(\Tilde{x}) + M_{h,in}(\Tilde{x}), \quad  T(0,\Tilde{x}) =  T_{s,in}(\Tilde{x}) + T_{h,in}(\Tilde{x}) ,  \quad \text{on} \ \Omega .
\end{equation}

 System (\ref{cross_sist}) -- (\ref{soinitdel2}) is very close to system $(1)$ of \cite{CdLFLM2023}, with the change of notations
$$ \Tilde{b} := \Tilde{\alpha}/\Tilde{\gamma}, \qquad \tilde{a} := \Tilde{\alpha}, \qquad   \tilde{d} := \Tilde{\mu}, $$
$$ \Tilde{e} := \Tilde{\beta}, \qquad \tilde{n} :=  \Tilde{\beta}/\Tilde{\eta},  \qquad   \tilde{\Gamma} := \Tilde{q}\, \Tilde{\gamma}. $$
 The only difference between (\ref{cross_sist}) -- (\ref{soinitdel2}) and $(1)$ of \cite{CdLFLM2023} lies in the much more 
complex shape of the diffusion terms for meso and top-predators in (\ref{cross_sist}) -- (\ref{soinitdel2}). Those are indeed 
cross diffusion terms, meaning that the diffusion rate for a given unknown can depend on the value of some other unknown.
 It can be noticed that the cross-diffusion terms of \eqref{cross_sist} are a convex combination of the diffusion coefficients $d_2^{(1)}$, $d_2^{(2)}$ for the meso-predator, $d_3^{(1)}$ and $d_3^{(2)}$ for the top-predator.
\medskip

Our first result concerns the stability analysis of system (\ref{cross_sist}) -- (\ref{soinitdel2}). As will be seen in Section \ref{tur_inst}, 
no Turing instability can appear for this system. This result is to be compared with the results of \cite{CdLFLM2023}, in which it is shown
that no Turing instability exists for a system similar to (\ref{cross_sist}) -- (\ref{soinitdel2}), but where diffusion terms have constant rates, 
and that Turing instability can appear when cross diffusion terms (with a significantly different structure from those of  (\ref{cross_sist}))
are considered. Our second result is a rigorous proof of convergence for suitable solutions of system  (\ref{lim_sist}) -- (\ref{soinit}) towards
solutions of system (\ref{lim_sist_del}) -- (\ref{soinitdel}) when $\varepsilon \to 0$ (and $\delta>0$ is fixed). A precise Theorem (Thm \ref{thtrun}) stating this convergence is written down (and proven) in Section \ref{rig_res}.
\medskip

We now briefly comment our results. First, up to our knowledge, it is the first time that a microscopic model leading to a cross diffusion system with three equations and two cross diffusion terms, in the macroscopic limit, is written down and analyzed. From the point of view of 
rigorous results, we point out that the passage to the limit when $\delta$ tends to $0$ in system (\ref{lim_sist_del}) -- (\ref{soinitdel})  (for large global solutions)
looks very difficult, since the theory of existence of large global solutions for systems like (\ref{cross_sist}) -- (\ref{soinitdel2})
(sometimes called ``non-triangular'' cross diffusion systems) has not yet been
developed. 

\section{Turing instability analysis}\label{tur_inst}

In this Section, we study the Turing instability regions associated to systems \eqref{cross_sist}-\eqref{BC_lim}. In order to do so, we first perform an adimensionalization, in order to keep a smaller number of parameters into the equations. Then, we investigate the possible linear instability of the coexistence (homogeneous) equilibrium $E^*$ (when it exists). Thus, we find the conditions guaranteeing the stability in absence of diffusion and those that guarantee the occurrence (or non-occurence) of Turing instability.

\subsection{Adimensionalization}
In order to simplify the notations and to keep only meaningful parameters, let us introduce the following dimensionless quantities
\begin{equation*}\label{adim}
        \begin{aligned}
            & P =: Ku, \quad M =: \frac{K\Tilde{r}}{\Tilde{\gamma}}v, \quad T =: \frac{\Tilde{m} \Tilde{r} K}{\Tilde{\gamma}}w, \quad \Tilde{t} =: \frac{t}{\Tilde{r}}, \quad \Tilde{\textbf{x}} =: L \textbf{x}, \\
            &b =: \frac{\Tilde{\gamma}}{\Tilde{\alpha}K},  \quad n =: \dfrac{\Tilde{\eta} \Tilde{\gamma}}{\Tilde{\beta} \Tilde{r} K}, \quad q =: \frac{\Gamma}{\Tilde{r}}, \quad s =: \frac{\Tilde{s}}{\Tilde{r}}, \quad d =: \frac{\Tilde{\mu}}{\Gamma}, \quad e =: \frac{\Tilde{m} \Tilde{\eta}}{\Gamma}, \quad c =: \dfrac{\Tilde{c} \Tilde{m} \Tilde{r}}{\Tilde{\alpha}}, \\
            &D_1 =: \frac{d_1}{\Tilde{r}L^2}, \quad D_2^{(1)} =: \frac{d_2^{(1)}}{\Tilde{r}L^2}, \quad D_2^{(2)} =: \frac{d_2^{(2)}}{\Tilde{r}L^2}, \quad D_3^{(1)} =: \frac{d_3^{(1)}}{\Tilde{r}L^2}, \quad D_3^{(2)} =: \frac{d_3^{(2)}}{\Tilde{r}L^2},
        \end{aligned}
    \end{equation*} 
where $L$ is the diameter of $\Omega$. 
System (\ref{cross_sist})
becomes
\begin{equation}
    \label{lim_sist_nodim}
        \begin{cases}
            \dfrac{\partial u}{\partial t} - D_1 \Delta u = u \left[ 1 - u - \dfrac{v}{b + u + cw} \right],  \\[10pt]
            \dfrac{\partial v}{\partial t} - \Delta \left( \dfrac{D_2^{(1)} (b + cw) + D_2^{(2)} u}{b+u+cw} v \right)  = qv \left( \dfrac{u}{b+u+cw} - d - { \dfrac{ew}{n+v}} \right),  \\[10pt] 
            \dfrac{\partial w}{\partial t} - {\Delta \left( \dfrac{D_3^{(1)} n  + D_3^{(2)} v}{n+v} w \right)}  = ws \left( 1 - \dfrac{w}{v} \right).
        \end{cases}
    \end{equation}
Let us append to the system \eqref{lim_sist_nodim}  the following homogeneous Neumann (no-flux) boundary conditions
\begin{equation}
    \label{BC_nondim}
        \nabla_x u\cdot\textbf{n}=0, \quad \nabla_x v\cdot\textbf{n}=0, \quad \nabla_x w \cdot \textbf{n}=0, \quad \text{on} \ \partial\Omega\times\R^+ .
\end{equation}

\subsection{Biological meaningful equilibria}

The biologically meaningful equilibria are the nonnegative solutions of the system
 \begin{equation}
    \label{sist_staz}
        \begin{cases}
            u \left[ 1 - u - \dfrac{v}{b + u + cw} \right] = 0, \\[10pt]
            qv \left( \dfrac{u}{b+u+cw} - d - \dfrac{ew}{n+v} \right) = 0,  \\[10pt] 
             ws \left( 1 - \dfrac{w}{v} \right) = 0.  
        \end{cases}
    \end{equation}
Avoiding the null solution, the solutions are the boundary equilibrium $E_1=\left(\dfrac{db}{1-d}, \dfrac{b(1-d-db)}{(1-d)^2}, 0 \right)$ describing the extinction of top-predator and defined when $d + d\,b <1$, and the interior equilibrium $E^*=(u^*,v^*,v^*)$ with non-null components representing the coexistence of the three populations.
We refer to \cite{CdLFLM2023} and \cite{FDV2020} for a description of the conditions
on parameters enabling for this interior equilibrium to exist.



For the sequel, we observe that (when the interior equilibrium exists) $w^*=v^*$, because of (\ref{sist_staz})$_3$.

\subsection{Preliminaries to stability}
Let us denote by ${E}^*=({u}^*, {v}^*, {w}^*)$ the coexistence equilibrium (when it exists). Introducing the perturbation fields 
    \begin{align*} 
        U = u - {u}^*,   \\ 
        V = v - {v}^*, \\
        W = w - {w}^*,
    \end{align*}
the system \eqref{lim_sist_nodim} becomes
    \begin{equation}\label{sist_pert}
    \begin{cases}
        \dfrac{\partial U }{\partial t} = a_{11} U  + a_{12} V + a_{13} W + D_1 \Delta U + \Tilde{F}_1, \\  

        \begin{aligned}
        \dfrac{\partial V }{\partial t} = a_{21} U + a_{22} V + a_{23} W &+ \Delta \left(  \dfrac{D_2^{(2)} - D_2^{(1)} }{ (b + {u}^* + c {w}^* )^2} (b + c {w}^*) {v}^* U \right) + \Delta \left( \dfrac{D_2^{(1)} (b + c {w}^*) + D_2^{(2)} {u}^*}{b + {u}^* + c {w}^*} V  \right) \\
        &+ \Delta \left( \dfrac{D_2^{(1)} - D_2^{(2)} }{ (b + {u}^* + c {w}^* )^2} c {u}^* {v}^* W \right) + \Tilde{F}_2,  
        \end{aligned} \\
                
        \dfrac{\partial W}{\partial t} = a_{31} U + a_{32} V + a_{33} W + \Delta \left( \dfrac{D_3^{(2)} - D_3^{(1)}}{(n + {v}^*)^2} n {w}^* V \right) + \Delta \left( \dfrac{D_3^{(1)} n  + D_3^{(2)} {v}^*}{n + {v}^*} W \right) + \Tilde{F}_3,
    \end{cases}
    \end{equation}
where $a_{ij}$ have the form given in \cite{CdLFLM2023}, namely
 \begin{equation}\label{aij}
\mkern-18mu
    \begin{aligned}
       & 
  a_{11} = 1 - 2u^* - \dfrac{v^*(b + c v^*)}{(b + u^* + c v^*)^2}; \quad 
            a_{12} = - \dfrac{u^*}{b + u^* + cv^*}; \quad a_{13}=\frac{cu^*v^*}{(b+u^*+cv^*)^2}; \\
    &  a_{21} = \dfrac{q v^*(b + cv^*)}{(b + u^* + cv^*)^2}; \quad 
            a_{22} =  \dfrac{qu^*}{b + u^* + cv^*}  - \dfrac{qen v^*}{(n + v^*)^2} -qd; \quad
            a_{23} = - \dfrac{cq u^* v^*}{b + u^* + cv^*}- \dfrac{eq v^*}{n+v^*}; \\
    &a_{31} = 0;\quad a_{32} = s\frac{({w}^*)^2}{({v}^*)^2} = s; \quad a_{33} = s\left( 1-2\frac{{w}^*}{{v}^*} \right) = -s;
    \end{aligned}
\end{equation}
and $\tilde{F}_{i}, \ (i=1,2,3)$ are the nonlinear terms. To  system \eqref{sist_pert} are appended the following no-flux boundary conditions
\begin{equation}\label{IC_pert}
    \nabla_x U\cdot\textbf{n}=0, \quad \nabla_x V \cdot \textbf{n}=0, \quad \nabla_x W \cdot \textbf{n}=0, \quad \text{on} \ \partial \Omega \times \R^+.
\end{equation}

\subsection{Linear instability in absence of diffusion}

When all diffusion rates are zero, 
 the linear system associated to system \eqref{sist_pert} evaluated around $E^*$ is
\begin{equation}\label{sist-matr}
        \dfrac{\partial \textbf{X}}{\partial t} = \mathcal{L}^0 \textbf{X},
    \end{equation}
    with
    \begin{equation}
    \label{10}
        \textbf{X} = (U,V,W), \quad  \mathcal{L}^0 = 
        \begin{pmatrix}
            a_{11} & a_{12} & a_{13} \\
            a_{21} & a_{22} & a_{23} \\
            0 & a_{32} & a_{33}  
        \end{pmatrix},
    \end{equation}
where $\textbf{X}$, $\mathcal{L}^0$ are the state vector and the Jacobian matrix evaluated in the coexistence equilibrium $E^*$, respectively.


Since $\mathcal{L}^0 = \mathcal{L}^0(E^*)$,
the parameters $a_{ij}$ are given by (\ref{aij}). 
\medskip

The characteristic equation associated to $\mathcal{L}^0$ is
\begin{equation}
    \lambda^3-I_1^0\lambda^2+I_2^0 \lambda-I_3^0=0,
\end{equation}
where $I_i^0$ ($i=1,2,3$) are the principal invariants of $\mathcal{L}^0$. The equilibrium $E^*$ is linearly stable if the Routh-Hurwitz conditions are verified \cite{Murray1_2002,Murray2_2002,Merkin1997}:
    \begin{equation}\label{RH}
        I_1^0<0, \qquad I_3^0<0, \qquad I_1^0I_2^0-I_3^0<0. 
    \end{equation}
The inequality $I_2^0>0$ is implied by \eqref{RH}. If at least one of \eqref{RH} is (strictly) reversed, $\mathcal{L}^0$ has at least one eigenvalue with positive real part, and hence, in that case, $E^*$ is linearly unstable.


Conditions on the parameters which ensure the stability of an existing interior equilibrium are presented in  \cite{CdLFLM2023}, 
cf. in particular eq. (31) and 
 Theorem $6.1$ in this work.

\subsection{Turing Instability when linear and cross diffusions are added}

In order to find  conditions guaranteeing the occurrence of Turing instability, we now consider linear and cross diffusions. In this case, the linear system associated to system \eqref{sist_pert} evaluated around $E^*$ is 
\begin{equation}\label{sist-matr2}
        \dfrac{\partial \textbf{X}}{\partial t} = \mathcal{L}^0 \textbf{X} + \mathcal{D} \Delta \textbf{X},
    \end{equation}
    with
    \begin{equation}
        \textbf{X} = (U,V,W), \quad  \mathcal{D} = 
        \begin{pmatrix}
            D_{1} & 0 & 0 \\
            l_{21} & l_{22} & l_{23} \\
            0 & l_{32} & l_{33}  
        \end{pmatrix},
    \end{equation}
where $\mathcal{L}^0$ is defined in \eqref{aij}, while $\mathcal{D}$ is the matrix of diffusion coefficients (remember that
$v^* = w^*$)
\begin{equation}
    \begin{gathered}
    \label{lij}
            l_{21} = \dfrac{D_2^{(2)} - D_2^{(1)}}{ (b + u^* + c v^* )^2} (b + c v^*) v^*; \quad 
            l_{22} = \dfrac{D_2^{(1)} (b + c v^*) + D_2^{(2)} u^*}{b + u^* + c v^*}; \quad l_{23} = \dfrac{D_2^{(1)} - D_2^{(2)} }{ (b + u^* + c v^* )^2} c u^* v^*; \\
            l_{32} = \dfrac{D_3^{(2)} - D_3^{(1)}}{(n + v^*)^2} n v^*; \quad
            l_{33} = \dfrac{D_3^{(1)} n  + D_3^{(2)} v^*}{n + v^*}. 
        \end{gathered}
\end{equation}


The dispersion relation governing the eigenvalues $\lambda$ in terms of wavenumber $k$ is now
    \begin{equation}\label{disp_rel}
        \lambda^3 - T_k(k^2) \lambda^2 + I_2(k^2) \lambda - h(k^2) = 0, 
    \end{equation}
    where $T_k(k^2)$, $h(k^2)$ and $I_2(k^2)$ are the principal invariants of the matrix $\mathcal{L}^0 - k^2 D$
 (note that $I_2(0) = I_2^0$).

We recall that according to Routh-Hurwitz condition, we know that all
     roots of the dispersion relation have negative real part if and only if 
        \begin{equation}
        \label{RH2}
            T_k(k^2) < 0, \quad h(k^2) < 0, \quad T_k(k^2) I_2(k^2) - h(k^2) < 0, \quad \forall k^2.
        \end{equation}
        (which implies $I_2(k^2) > 0$). Turing instability occurs when $E^*$ is linearly stable in absence of diffusion and there exists at least one value of $k^2$, such that at least one of conditions \eqref{RH2} is reversed.
We observe that
\begin{equation}
    \begin{aligned}
        T_k(k^2) &= \text{tr}(\mathcal{L}^0 - k^2 D) = - k^2 \text{tr}(D) + I_1^0 = \underbrace{I_1^0}_{-} - k^2 \underbrace{(D_1 + l_{22} + l_{33})}_{+} <0 \\
        I_2(k^2) &= k^4 (\underbrace{D_1 l_{22}}_{+} + \underbrace{D_1 l_{33}}_{+} + \underbrace{l_{22} l_{33}}_{+} - \underbrace{l_{23} l_{32}}_{-})  - k^2 [ \underbrace{D_1 (a_{22} - s)}_{-} + \underbrace{l_{22}(a_{11}-s)}_{-}  \\
        &+\underbrace{l_{33}(a_{11}+ a_{22})}_{-} - \underbrace{l_{21} a_{12}}_{+} - \underbrace{l_{32} a_{23}}_{+} - \underbrace{l_{23}s}_{+}]  + \underbrace{I_2^0}_{+} > 0\\
        h(k^2) = \det(\mathcal{L}^0 - k^2 \mathcal{D}) &= -k^6 \underbrace{\det( \mathcal{D} )}_{+} + k^4[\underbrace{a_{11} l_{22} l_{33}}_{-} + \underbrace{a_{22}D_1 l_{33}}_{-} - \underbrace{D_1 l_{22}s}_{+} - \underbrace{a_{11} l_{23} l_{32}}_{+} - \underbrace{D_1 l_{23} s}_{+} - \underbrace{D_1 l_{32} a_{23}}_{+} \\
        &- \underbrace{l_{33} l_{21} a_{12}}_{+} + \underbrace{l_{21} l_{32} a_{13}}_{+}] + k^2 [ \underbrace{D_1(a_{22} s + a_{23} s}_{-}) + \underbrace{l_{22} a_{11} s}_{-} - \underbrace{l_{33} a_{11} a_{22}}_{+} \\
        &+ \underbrace{l_{33} a_{12} a_{21}}_{-} -\underbrace{a_{13} l_{21}s}_{-} - \underbrace{a_{13} l_{32} a_{21}}_{-} + \underbrace{a_{11} l_{23}s}_{-} + \underbrace{a_{11} l_{32} a_{23}}_{-} - \underbrace{s l_{21} a_{12}}_{+} ] + \underbrace{I_3^0}_{-}  \\
        &= -k^6 \det(\mathcal{D}) + k^4[ a_{11} l_{22} l_{33} + a_{22}D_1 l_{33} - D_1 l_{22}s - a_{11} l_{23} l_{32} - D_1 l_{23} s - D_1 l_{32} a_{23}  \\
        &+ (-l_{33} l_{21} a_{12} + l_{21} l_{32} a_{13})] + k^2 [ D_1(a_{22} s + a_{23} s) + l_{22} a_{11} s - l_{33} a_{11} a_{22} \\
        & + (l_{33} a_{12} a_{21} - a_{13} l_{32} a_{21}) + (- a_{13} l_{21}s - s l_{21} a_{12}) + a_{11} l_{23}s + a_{11} l_{32} a_{23} ] + I_3^0
    \end{aligned}
    \end{equation}

    Since $D_1$ and all $D_{i}^j, i=2,3; j=1,2$ are strictly positive, then $l_{22} > 0$, $l_{33} > 0$.
    
We expect the diffusion rate $D_{2,3}^{(2)}$ of handling predators to be smaller than the diffusion rate of searching predators $D_{2,3}^{(1)}$. Therefore, since $D_2^{(2)} < D_2^{(1)} \Longrightarrow l_{21} < 0$, $l_{23} > 0$, and similarly, since $D_3^{(2)} < D_3^{(1)} \Longrightarrow l_{32} < 0$.

The terms $l_{21}l_{32}a_{13}$, $a_{13}l_{21}s$ and $a_{13}l_{32}a_{21}$ make the sign of $h(k^2)$ undefined. However, since (remember that $v^* = w^*$)
    \begin{equation}
        l_{33} > \lvert{l_{32}\rvert} \quad  \text{ and } \quad \lvert{a_{12}\rvert} > a_{13}
    \end{equation}


    and remembering the sign of each $a_{ij}$, we get
    \begin{equation}
    \begin{aligned}
        & -l_{33} l_{21} a_{12} + l_{21} l_{32} a_{13} = - \lvert{l_{21}\rvert} \bigl[ l_{33}  \lvert{a_{12}\rvert} - a_{13}  \lvert{l_{32}\rvert} \bigr] <  - \lvert{l_{21}\rvert} \bigl[ l_{33}  \lvert{a_{12}\rvert} - l_{33} a_{13}\bigr] < 0 \\[10pt]
        & l_{33} a_{12} a_{21} - a_{13} l_{32} a_{21} = - a_{21} \bigl[ - a_{12} l_{33} + a_{13} l_{32}\bigr] < - a_{21} \bigl[ -\lvert{l_{32} \rvert} a_{12} + \lvert{a_{12}\rvert} l_{32} \bigr] = 0 \\[10pt]
        & - a_{13} l_{21}s - s l_{21} a_{12} = s \lvert{l_{21}\rvert} \bigl[a_{13} + a_{12}\bigr] < s \lvert{l_{21}\rvert} \bigl[\lvert{a_{12}\rvert} + a_{12}\bigr] = 0. \\
    \end{aligned}
    \end{equation}

We obtain that the conditions \eqref{RH2} are verified, and no one is reversed. So, 
Turing instability can't occur.

\section{Rigorous results for the passage to the limit in microscopic models}\label{rig_res}



We start from system \eqref{lim_sist} -- \eqref{soinit} in a smooth bounded open subset $\Omega$ of $\R^N$ and consider the limit when $\varepsilon \to 0$, and when $\delta>0$ is fixed. In the notation of the concentrations, we drop therefore the superscript $\delta$, so that the system writes (from now on, we write $t$ instead of $\Tilde{t}$ for the sake of simplicity)

\begin{equation}
\label{lim_sist_rig}
        \begin{cases}
            \dfrac{\partial P^{\varepsilon}}{\partial {t}} = \Tilde{r} \left(1 - \dfrac{P^{\varepsilon}}{K} \right) P^{\varepsilon} - \dfrac{\Tilde{\alpha} P^{\varepsilon} M_s^{\varepsilon}}{1 + \Tilde{c}\, \left(T_s^{\varepsilon}+T_h^{\varepsilon}\right)} + d_1 \Delta P^{\varepsilon}, \\[10pt]
            \dfrac{\partial M_s^{\varepsilon}}{\partial {t}} = \dfrac{1}{\delta} \left[  - \dfrac{\Tilde{\alpha} P^{\varepsilon} M_s^{\varepsilon}}{1 + \Tilde{c}\, \left(T_s^{\varepsilon}+T_h^{\varepsilon}\right)} + \Tilde{\gamma} M_h^{\varepsilon} \right] + \Gamma M_h^{\varepsilon} - \Tilde{\mu} M_s^{\varepsilon} - \Tilde{\beta} M_s^{\varepsilon} \, T_s^{\varepsilon} + d_2^{(1)} \Delta M_s^{\varepsilon}, \\[10pt]
            \dfrac{\partial M_h^{\varepsilon}}{\partial {t}} = \dfrac{1}{\delta} \left[  \dfrac{\Tilde{\alpha} P^{\varepsilon} M_s^{\varepsilon}}{1 + \Tilde{c}\, \left(T_s^{\varepsilon}+T_h^{\varepsilon}\right)} - \Tilde{\gamma} M_h^{\varepsilon} \right] - \Tilde{\mu} M_h^{\varepsilon} - \Tilde{\beta} M_h^{\varepsilon} \, T_s^{\varepsilon} + d_2^{(2)} \Delta M_h^{\varepsilon}, \\[10pt]
            \dfrac{\partial T_s^{\varepsilon}}{\partial {t}} = \dfrac{1}{\varepsilon} \left[ - \Tilde{\beta} (M_s^{\varepsilon} + M_h^{\varepsilon}) T_s^{\varepsilon} + \Tilde{\eta} T_h^{\varepsilon} \right] + \Tilde{s} \,T_s^{\varepsilon}\,  \left(1 - \dfrac{T_h ^{\varepsilon} + T_s^{\varepsilon} }{ \Tilde{m}(M_s^{\varepsilon} + M_h^{\varepsilon})} \right) + d_3^{(1)} \Delta T_s^{\varepsilon}, \\[10pt]
            \dfrac{\partial T_h^{\varepsilon}}{\partial {t}} = \dfrac{1}{\varepsilon} \left[  \Tilde{\beta} (M_s^{\varepsilon} + M_h^{\varepsilon}) T_s^{\varepsilon} - \Tilde{\eta} T_h^{\varepsilon} \right] + \Tilde{s} \,T_h^{\varepsilon}\,  \left(1 - \dfrac{T_h ^{\varepsilon} + T_s^{\varepsilon} }{ \Tilde{m}(M_s^{\varepsilon} + M_h^{\varepsilon})} \right) + d_3^{(2)} \Delta T_h^{\varepsilon}, \\
        \end{cases}
    \end{equation}
    together with homogeneous Neumann boundary conditions
    \begin{equation}
    \label{titii}
        \nabla_x P^{\varepsilon}\cdot\textbf{n}=0, \quad \nabla_x M_s^{\varepsilon}\cdot\textbf{n}=0, \quad  \nabla_x M_h^{\varepsilon}\cdot\textbf{n}=0\quad \nabla_x T_s^{\varepsilon}\cdot\textbf{n}=0,\quad \nabla_x T_h^{\varepsilon}\cdot\textbf{n}=0, 
    \end{equation}
    for all $t \in \R^{+}$, $\textbf{x} \in \partial \Omega$, where $\textbf{n} := n(\textbf{x})$ denotes the exterior normal to $\Omega$ at a point $\textbf{x} \in \partial \Omega$, and with initial data (for $x \in \Omega$)
    \begin{equation}
    \label{titi}
    \begin{aligned}
        & \quad \quad \quad \quad P^{\varepsilon}(0,x) =  P_{in}(x), \quad M_s^{\varepsilon}(0,x) =  M_{s, in}(x), \\
        & M_h^{\varepsilon}(0,x) =  M_{h, in}(x) \quad T_s^{\varepsilon}(0,x) =  T_{s, in}(x), \quad T_h^{\varepsilon}(0,x) =  T_{h, in}(x).
    \end{aligned}
    \end{equation}

We write down a theorem stating what is the limit of the solutions of system \eqref{lim_sist_rig} -- \eqref{titi}
 when $\varepsilon \to 0$.
\medskip

\begin{thm} \label{thtrun}
    Let $\Omega$ be a smooth bounded open subset of $\R^N$ for $N=1,2$, let $d_1 , d_2^{(1)}, d_2^{(2)}, d_3^{(1)}, d_3^{(2)} > 0$ be diffusion rates, $\Tilde{\alpha}, \Tilde{\beta}, \Tilde{\gamma}, \Gamma, \Tilde{\eta}, \Tilde{\mu}, \Tilde{c}, K, \Tilde{m}, \Tilde{r}, \Tilde{s}>0, \delta > 0$ be parameters, and $P_{in} :=  P_{in}(x) \ge 0, M_{s, in} :=  M_{s, in}(x) \ge D > 0, M_{h, in} :=  M_{h, in}(x) \ge D >0,$ and $T_{s, in} :=  T_{s, in}(x) \ge 0, T_{h, in} :=  T_{h, in}(x) \ge 0$ be  initial data
lying in $C^2(\overline{\Omega})$ and compatible with Neumann boundary condition. 
\par
    Then for each $\varepsilon > 0$, there exists a unique global  (for $t \ge 0$) classical solution $(P^{\varepsilon}, M_s^{\varepsilon}, M_h^{\varepsilon}, T_s^{\varepsilon}, T_h^{\varepsilon})$ to system \eqref{lim_sist_rig}-\eqref{titi}, such that $P^{\varepsilon}, M_s^{\varepsilon}, M_h^{\varepsilon}$ are nonnegative, and $M_s^{\varepsilon}, M_h^{\varepsilon} \ge D_T >0$ on $[0,T]$ (with $D_T$ depending only on $T$ and the parameters of the system, including $\delta$, but not $\varepsilon$) for all $T>0$.
    \par 
Moreover, when $\varepsilon \to 0$, one can extract from $P^{\varepsilon}, M_s^{\varepsilon}, M_h^{\varepsilon}$ a subsequence which is bounded in $L^{\infty}([0,T] \times \Omega)$ for all $T > 0$ and converges a.e. towards  functions $P, M_s, M_h$ respectively, lying in $L^{\infty}([0,T] \times \Omega)$ for all $T>0$, and one can extract from $T_s^{\varepsilon}$ and $T_h^{\varepsilon}$  a subsequence which converges strongly (and a.e.) in $L^{2 + \zeta}([0,T] \times \Omega)$ towards functions $T_s$ and $T_h$ respectively, lying in $L^{2 + \zeta}([0,T] \times \Omega)$, for all $T > 0$ and some $\zeta > 0$. 
    Moreover, $M_s, M_h \ge D_T > 0$ (for any $T > 0$). \\
    Finally, $P, M_s, M_h, T_s, T_h$ are ``very-weak'' solutions of the reaction-cross diffusion system
    \begin{equation}
    \label{senza eps}
        \begin{cases}
            \dfrac{\partial P}{\partial t} = \Tilde{r} \left(1 - \dfrac{P}{K} \right) P - \dfrac{\Tilde{\alpha} P M_s}{1 + \Tilde{c}\, \left(T_s+T_h \right)}  + d_1 \Delta P, \\[10pt]
            \dfrac{\partial M_s}{\partial t} = \dfrac{1}{\delta} \left[  - \dfrac{\Tilde{\alpha} P M_s}{1 + \Tilde{c}\, \left(T_s+T_h \right)} + \Tilde{\gamma} M_h \right] + \Gamma M_h - \Tilde{\mu} M_s - \Tilde{\beta} M_s\, T_s + d_2^{(1)} \Delta M_s, \\[10pt]
            \dfrac{\partial M_h}{\partial t} = \dfrac{1}{\delta} \left[  \dfrac{\Tilde{\alpha} P M_s}{1 + \Tilde{c}\, \left(T_s+T_h \right)} - \Tilde{\gamma} M_h \right] - \Tilde{\mu} M_h - \Tilde{\beta} M_h \, T_s + d_2^{(2)} \Delta M_h ,\\[10pt]
            \dfrac{\partial (T_h + T_s)}{\partial t} =  \Tilde{s} \,(T_s + T_h) \,  \left(1 - \dfrac{(T_s + T_h)}{ \Tilde{m}(M_s + M_h) } \right) + \Delta \left( d_3^{(1)}  T_s +  d_3^{(2)}  T_h \right) ,\\[10pt]
            \Tilde{\beta} (M_s + M_h) T_s =  \Tilde{\eta} T_h , \\ 
        \end{cases}
    \end{equation}
    together with the homogeneous Neumann boundary conditions
    \begin{equation}
    \label{CBo}
       \nabla_x P\cdot\textbf{n}=0, \quad \nabla_x M_s\cdot\textbf{n}=0, \quad  \nabla_x M_h\cdot\textbf{n}=0\quad \nabla_x (T_s + T_h) \cdot\textbf{n}=0, 
    \end{equation}
    for all $t \in \R_+$, $\textbf{x}$ in $\partial \Omega$, 
and with initial data (for $x \in \Omega$)
    \begin{equation} 
    \label{titi2}
       \begin{aligned}
           & P(0,x) =  P_{in}(x), \quad M_s(0,x) = M_{s, in}(x),\quad \\ M_h(0,x) = & M_{h, in}(x), \quad (T_s + T_h)(0,x) = T_{s, in}(x) +  T_{h, in}(x).
       \end{aligned}
    \end{equation}
 By ``very-weak'' solution, we mean that the last identity of system \eqref{senza eps}, that is
   \begin{equation}\label{exf}
 \Tilde{\beta} (M_s + M_h) T_s =  \Tilde{\eta} T_h,
\end{equation}
 holds a.e., and that for all $\xi, \varphi, \chi, \psi \in {C}_c^2([0, +\infty) \times \Bar{\Omega})$ such that $\nabla_x \xi \cdot \textbf{n} |_{\partial \Omega} = 0, \quad \nabla_x \varphi \cdot \textbf{n} |_{\partial \Omega} = 0, \quad \nabla_x \chi \cdot \textbf{n} |_{\partial \Omega} = 0, \quad \nabla_x \psi \cdot \textbf{n} |_{\partial \Omega} = 0, $ the following identities hold:
    
    \begin{equation}  
    \label{1}
        - \int_0^{\infty} \int_{\Omega} P \partial_t \varphi - \int_{\Omega} \varphi(0, \cdot) P_{in} - d_1 \int_0^{\infty} \int_{\Omega} P \Delta \varphi = \int_0^{\infty} \int_{\Omega} \left[ \Tilde{r} \left(1 - \dfrac{P}{K} \right) P -  \dfrac{\Tilde{\alpha} P M_s}{1 + \Tilde{c}\, \left(T_s+T_h \right)}  \right] \varphi, 
    \end{equation}
    \begin{equation}
    \label{2}
        \begin{aligned}
            - \int_0^{\infty} \int_{\Omega} M_s \partial_t \psi - \int_{\Omega} \psi(0, \cdot) M_{s,in} - \int_0^{\infty} \int_{\Omega} d_2^{(1)} M_s \Delta \psi  &= \int_0^{\infty} \int_{\Omega}  \left[ \Gamma M_h - \Tilde{\mu} M_s - \Tilde{\beta} M_s T_s \right] \psi \\
            &+ \int_0^{\infty} \int_{\Omega} \dfrac{1}{\delta} \left[  - \dfrac{\Tilde{\alpha} P M_s}{1 + \Tilde{c}\, \left(T_s+T_h \right)} + \Tilde{\gamma} M_h \right] \psi, \\[10pt]
        \end{aligned}
    \end{equation}

       \begin{equation}
       \label{3}
           \begin{aligned}
               - \int_0^{\infty} \int_{\Omega} M_h \partial_t \chi - \int_{\Omega} \chi(0, \cdot) M_{h,in}- \int_0^{\infty} \int_{\Omega} d_2^{(2)} M_h \Delta \chi &= \int_0^{\infty} \int_{\Omega}  \left[ - \Tilde{\mu} M_h - \Tilde{\beta} M_h T_s \right] \chi \\
               &+ \int_0^{\infty} \int_{\Omega} \dfrac{1}{\delta} \left[ \dfrac{\Tilde{\alpha} P M_s}{1 + \Tilde{c}\, \left(T_s+T_h \right)} - \Tilde{\gamma} M_h \right] \chi,
           \end{aligned}
       \end{equation}

    \begin{equation} \label{tteq}
        \begin{aligned}
            - \int_0^{\infty} \int_{\Omega} (T_s + T_h) \partial_t \xi - \int_{\Omega} \xi(0, \cdot) (T_{s, in} + T_{h,in}) &- \int_0^{\infty} \int_{\Omega} \left( d_3^{(1)}  T_s +  d_3^{(2)}  T_h \right) \Delta \xi = \\
            &= \int_0^{\infty} \int_{\Omega}  \left[ \Tilde{s} \,(T_s + T_h)\,  \left(1 - \dfrac{(T_s + T_h)}{ \Tilde{m} (M_h + M_s)} \right) \right] \xi.
        \end{aligned}
    \end{equation}

    Note that the reaction-diffusion system \eqref{senza eps} -- \eqref{titi2} 
 can be rewritten in the simpler form (with $T_* : = T_s + T_h$)
    \begin{equation}
    \label{simpform}
        \begin{cases}
            \dfrac{\partial P}{\partial t} = \Tilde{r} \left(1 - \dfrac{P}{K} \right) P - \dfrac{\Tilde{\alpha} \Tilde{\gamma} P(M_s + M_h)}{\Tilde{\gamma} + \Tilde{\alpha} P + \tilde{c} \tilde{\gamma}T_*} + d_1 \Delta P, \\[10pt]
            
            \dfrac{\partial M_s}{\partial t} = \dfrac{1}{\delta} \left[  - \dfrac{\Tilde{\alpha} P M_s}{1 + \Tilde{c}\, T_*} + \Tilde{\gamma} M_h \right] + \Gamma M_h - \Tilde{\mu} M_s -   \dfrac{ \Tilde{\beta} \Tilde{\eta} M_s T_*}{\Tilde{\eta} + \Tilde{\beta}(M_s + M_h)} + d_2^{(1)} \Delta M_s ,\\[10pt]
            \dfrac{\partial M_h}{\partial t} = \dfrac{1}{\delta} \left[  \dfrac{\Tilde{\alpha} P M_s}{1 + \Tilde{c}\, T} - \Tilde{\gamma} M_h \right] - \Tilde{\mu} M_h -  \dfrac{\Tilde{\beta} \Tilde{\eta}  M_h T_*}{\Tilde{\eta} + \Tilde{\beta} (M_s + M_h)} + d_2^{(2)} \Delta M_h, \\[10pt]
            
              \dfrac{\partial T_*}{\partial t} = \Tilde{s} \left(1 - \dfrac{T_*}{\Tilde{m} (M_h + M_s)} \right) T_*  + {\Delta \left( \dfrac{(d_3^{(1)} \Tilde{\eta} + d_3^{(2)} \Tilde{\beta}(M_h + M_s))}{\Tilde{\eta} + \Tilde{\beta} (M_h + M_s)} T_* \right)},
        \end{cases}
    \end{equation}
    with the following boundary conditions 
    \begin{equation}
    \nabla_x P\cdot\textbf{n}=0, \quad \nabla_x M_s\cdot\textbf{n}=0, \quad  \nabla_x M_h\cdot\textbf{n}=0\quad \nabla_x T_* \cdot\textbf{n}=0,
    \end{equation}
 for all $t \in \R_+$, $\textbf{x}$ in $\partial \Omega$, 
and initial data (for $x\in \Omega$)
    \begin{equation} 
    \label{titi3}
       \begin{aligned}
           & P(0,x) =  P_{in}(x), \quad M_s(0,x) = M_{s, in}(x),\quad \\ 
           M_h(0,x) &=  M_{h, in}(x), \quad T_*(0,x) = T_{s, in}(x) +  T_{h, in}(x).
       \end{aligned}
    \end{equation}
\end{thm}

\begin{proof}
{\bf{First step}}: Existence of smooth solutions to system  \eqref{senza eps} -- \eqref{titi2}
\medskip

 We recall that $\Omega$ is a smooth open bounded subset of $\R^N$ ($N =1$ or $2$) and $\varepsilon >0$ is given. Our aim in this step is to prove that there exists a (unique) classical solution to  system \eqref{lim_sist_rig} -- \eqref{titi}.

For any $\nu \in ]0,1[$, we introduce the following approximating system (we use a notation in which the dependence with respect to $\varepsilon$ of the concentrations is dropped, since $\varepsilon>0$ is given, and in which the dependence with respect to $\nu$ is explicitly written):

 \begin{equation}
\label{lim_sist_appr}
\mkern-36mu
        \begin{cases}
            \dfrac{\partial P^{\nu}}{\partial t} = \Tilde{r} \left(1 - \dfrac{P^{\nu}}{K\, (1 +\nu [P^{\nu}]^2)} \right) \, \dfrac{P^{\nu}}{1 + \nu [P^{\nu}]^2} - \dfrac{\Tilde{\alpha} \dfrac{P^{\nu}}{1 + \nu [P^{\nu}]^2}  \dfrac{M_s^{\nu}}{1 + \nu [M_s^{\nu}]^2} }{1 + \Tilde{c}\, \sqrt{ \nu^2 + \left(T_s^{\nu}+T_h^{\nu}\right)^2 } } + d_1 \Delta P^{\nu} ,\\[30pt]
            \begin{aligned}
            \dfrac{\partial M_s^{\nu}}{\partial t} = \dfrac{1}{\delta} \left[  - \dfrac{\Tilde{\alpha} \dfrac{P^{\nu}}{1 + \nu [P^{\nu}]^2}  \dfrac{M_s^{\nu}}{1 + \nu [M_s^{\nu}]^2} }{1 + \Tilde{c}\, \sqrt{ \nu^2 + \left(T_s^{\nu}+T_h^{\nu}\right)^2 } } + \Tilde{\gamma} \dfrac{M_h^{\nu}}{1 + \nu [M_h^{\nu}]^2} \right] + \Gamma &\dfrac{M_h^{\nu}}{1 + \nu [M_h^{\nu}]^2} - \Tilde{\mu} \dfrac{M_s^{\nu}}{1 + \nu [M_s^{\nu}]^2} \\
            &- \Tilde{\beta} \dfrac{M_s^{\nu}}{1 + \nu [M_s^{\nu}]^2} \, \dfrac{T_s^{\nu}}{1 + \nu [T_s^{\nu}]^2} + d_2^{(1)} \Delta M_s^{\nu},
            \end{aligned} \\[30pt]
            \dfrac{\partial M_h^{\nu}}{\partial t} = \dfrac{1}{\delta} \left[  \dfrac{\Tilde{\alpha} \dfrac{P^{\nu}}{1 + \nu [P^{\nu}]^2} \dfrac{M_s^{\nu}}{1 + \nu [M_s^{\nu}]^2} }{1 + \Tilde{c}\, \sqrt{ \nu^2 + \left(T_s^{\nu}+T_h^{\nu}\right)^2 } } - \Tilde{\gamma} \dfrac{M_h^{\nu}}{1 + \nu [M_h^{\nu}]^2} \right] - \Tilde{\mu} \dfrac{M_h^{\nu}}{1 + \nu [M_h^{\nu}]^2} - \Tilde{\beta} \dfrac{M_h^{\nu}}{1 + \nu [M_h^{\nu}]^2} \, \dfrac{T_s^{\nu}}{1 + \nu [T_s^{\nu}]^2} + d_2^{(2)} \Delta M_h^{\nu} ,\\[30pt]
            \begin{aligned}
            \dfrac{\partial T_s^{\nu}}{\partial t} = \dfrac{1}{\varepsilon} \biggl[ - \Tilde{\beta} \left(\dfrac{M_s^{\nu}}{1 + \nu [M_s^{\nu}]^2} + \dfrac{M_h^{\nu}}{1 + \nu [M_h^{\nu}]^2}\right)& \dfrac{T_s^{\nu}}{1 + \nu [T_s^{\nu}]^2} + \Tilde{\eta} \dfrac{T_h^{\nu}}{1 + \nu [T_h^{\nu}]^2} \biggr] \\
            &+ \Tilde{s} \,\dfrac{T_s^{\nu}}{1 + \nu [T_s^{\nu}]^2}\,  \left(1 - \dfrac{\dfrac{T_h^{\nu}}{1 + \nu [T_h^{\nu}]^2} + \dfrac{T_s^{\nu}}{1 + \nu [T_s^{\nu}]^2} }{ \Tilde{m} \sqrt{\nu^2 + (M_s^{\nu} + M_h^{\nu})^2} } \right) + d_3^{(1)} \Delta T_s^{\nu}, 
            \end{aligned} \\[50pt]
            \begin{aligned}
            \dfrac{\partial T_h^{\nu}}{\partial t} = \dfrac{1}{\varepsilon} \biggl[  \Tilde{\beta} \left(\dfrac{M_s^{\nu}}{1 + \nu [M_s^{\nu}]^2} + \dfrac{M_h^{\nu}}{1 + \nu [M_h^{\nu}]^2}\right)& \dfrac{T_s^{\nu}}{1 + \nu [T_s^{\nu}]^2} - \Tilde{\eta} \dfrac{T_h^{\nu}}{1 + \nu [T_h^{\nu}]^2}
            \biggr] \\
            &+ \Tilde{s} \,\dfrac{T_h^{\nu}}{1 + \nu [T_h^{\nu}]^2}\,   \left(1 - \dfrac{\dfrac{T_h^{\nu}}{1 + \nu [T_h^{\nu}]^2} + \dfrac{T_s^{\nu}}{1 + \nu [T_s^{\nu}]^2} }{ \Tilde{m} \sqrt{\nu^2 + (M_s^{\nu} + M_h^{\nu})^2} } \right) + d_3^{(2)} \Delta T_h^{\nu},
            \end{aligned} \\
        \end{cases}
    \end{equation}
with Neumann boundary conditions and initial data \eqref{titi} (that is, independent of $\nu$). 
\medskip


In order to prove that the system admits a global smooth solution, it is sufficient, using for example the results of \cite{LD_milano}, to show that the reaction part of each right-hand side of \eqref{lim_sist_appr} is bounded and globally Lipschitz (in terms of the unknowns). 

\smallskip

By construction, the reaction part of the right-hand side of system (\ref{lim_sist_appr}) is bounded. Let us show that it is globally Lipschitz in terms of the concentrations. 
If we call $f_i = f_i(P^{\nu}, M_s^{\nu}, M_h^{\nu}, T_s^{\nu}, T_h^{\nu})$, with $i = 1, \dots, 5$  the reaction part of the right-hand side of equations \eqref{lim_sist_appr}, we need to prove that, for all $i = 1, \dots 5$ (and some constant $C>0$ which may depend on $\nu$, $\varepsilon$, and the other parameters of the system)

\begin{equation}
\label{GL}
    \begin{aligned}
       & \lvert f_i(P^{\nu}, M_s^{\nu}, M_h^{\nu}, T_s^{\nu}, T_h^{\nu}) - f_i(P^{\nu'}, M_s^{\nu'}, M_h^{\nu'}, T_s^{\nu'}, T_h^{\nu'}) \rvert \leq \\ \leq C ( \lvert P^{\nu} - & P^{\nu'} \rvert + \lvert M_s^{\nu} - M_s^{\nu'} \rvert +\lvert M_h^{\nu} - M_h^{\nu'}\rvert +\lvert T_s^{\nu} - T_s^{\nu'}\rvert + \lvert T_h^{\nu} - T_h^{\nu'} \rvert ) .
    \end{aligned}
\end{equation}


Hence, in order to obtain \eqref{GL}, it is sufficient to show that all the derivatives $\dfrac{\partial f_i}{\partial P^{\nu}}, \dfrac{\partial f_i}{\partial M_s^{\nu}}, \dfrac{\partial f_i}{\partial M_h^{\nu}}, \dfrac{\partial f_i}{\partial T_s^{\nu}}, \dfrac{\partial f_i}{\partial T_h^{\nu}}$, with $i = 1, \dots, 5$ are bounded (by some constant $C>0$ which may depend on $\nu$, $\varepsilon$, and the other parameters of the system).

Let us consider first  the terms $ \dfrac{\partial f_1}{\partial P^{\nu}}$ and $ \dfrac{\partial f_1}{\partial T_s^{\nu}} $. We 
estimate
    \begin{equation}
    \begin{aligned}
        \left| \dfrac{\partial f_1}{\partial P^{\nu}} \right| &= \left| \tilde{r} \dfrac{ 1 - \nu [P^{\nu}]^2}{(1 + \nu [P^{\nu}]^2)^2} - \dfrac{2 \tilde{r}}{K} \dfrac{(1 - \nu [P^{\nu}]^2) P^{\nu}}{(1 + \nu [P^{\nu}]^2)^3} - \dfrac{\tilde{\alpha} M_s^{\nu}}{(1 + \nu [M_s^{\nu}]^2)} \dfrac{1}{ [1 + \tilde{c} \sqrt{\nu^2 + (T_s^{\nu}+ T_h^{\nu})^2 }]} \dfrac{1 - \nu [P^{\nu}]^2}{(1 + \nu [P^{\nu}]^2)^2} \right| \\
        &\leq  \tilde{r}\, \left( 1 + \dfrac{1}{K \sqrt{\nu}} \right) +  \dfrac{\Tilde{\alpha} }{ 2 \sqrt{\nu}} , \\
        \left| \dfrac{\partial f_1}{\partial T_s^{\nu}} \right| &= \left| \dfrac{\tilde{\alpha} P^{\nu}}{(1 + \nu [P^{\nu}]^2)} \dfrac{M_s^{\nu}}{(1 + \nu [M_s^{\nu}]^2)} \dfrac{\Tilde{c}}{ [1 + \tilde{c} \sqrt{\nu^2 + (T_s^{\nu}+ T_h^{\nu})^2 }]^2 }\dfrac{(T_s^{\nu} + T_h^{\nu})}{ \sqrt{\nu^2 + (T_s^{\nu}+ T_h^{\nu})^2 }} \right| \leq \dfrac{\tilde{\alpha}\Tilde{c}}{ 4\nu},
    \end{aligned}
    \end{equation}
since 
$$ \left| \frac{x}{1 + \nu x^2} \right| \le \frac1{2 \sqrt{\nu} }, \qquad \frac{|x|}{\sqrt{\nu^2  + x^2} } \le 1 . $$

All the other terms (partial derivatives of $f_i$ for  $i = 1, \dots, 5$) can be treated similarly.

    Therefore, the approximating  system (\ref{lim_sist_appr}) admits (at least) a global smooth solution (cf. \cite{LD_milano}). Moreover, the structure of the system ensures that the 
unknowns are nonnegative, that is $P^{\nu}, M_s^{\nu}, M_h^{\nu}, T_s^{\nu}, T_h^{\nu} \ge 0$.

\smallskip

 We now establish a few $\nu$-independent a priori estimates for this system.
Using the nonnegativity of the unknowns, we see that 

 \begin{equation}
\label{lim_sist_ineg}
        \begin{cases}
         \dfrac{\partial P^{\nu}}{\partial t} \le \Tilde{r}  P^{\nu} + d_1 \Delta P^{\nu}, \\[10pt]
            \dfrac{\partial M_s^{\nu}}{\partial t} \le \dfrac{\Tilde{\gamma}  }{\delta}  M_h^{\nu}   + \Gamma M_h^{\nu} + d_2^{(1)} \Delta M_s^{\nu} ,\\[10pt]
            \dfrac{\partial M_h^{\nu}}{\partial t} \le \dfrac{\Tilde{\alpha} }{\delta}   P^{\nu} M_s^{\nu} + d_2^{(2)} \Delta M_h^{\nu}, \\[10pt]
            \dfrac{\partial T_s^{\nu}}{\partial t} \le \dfrac{\Tilde{\eta}}{\varepsilon} T_h^{\nu}+ \Tilde{s} \,T_s^{\nu} + d_3^{(1)} \Delta T_s^{\nu} ,\\[10pt]
            \dfrac{\partial T_h^{\nu}}{\partial t}\le  \dfrac{\Tilde{\beta}}{\varepsilon} (M_s^{\nu} + M_h^{\nu}) T_s^{\nu} + \Tilde{s} \,T_h^{\nu} + d_3^{(2)} \Delta T_h^{\nu}. \\
        \end{cases}
    \end{equation}

Multiplying by $ (P^{\nu})^p$ (for some $p\ge 1$) inequality (\ref{lim_sist_ineg})$_{1}$ and integrating by parts, we get
\begin{equation}
    \partial_t \int_{\Omega} \frac{(P^{\nu})^{p+1}}{p+1} \leq  \Tilde{r} \int_{\Omega} ( P^{\nu})^{p+1} - d_1p \int_{\Omega} ( P^{\nu})^{p-1} |\nabla P^{\nu}|^2 ,
\end{equation}
so that for all $t \ge 0$, using Gronwall's lemma, 
\begin{equation}
    \lVert P^{\nu}(t,\cdot) \rVert_{L^{p+1}} \le  \lVert P^{\nu}(0,\cdot) \rVert_{L^{p+1}}\, e^{  \Tilde{r}  \, t} ,
\end{equation}
and letting $p \to \infty$, we end up with the (maximum principle) estimate
\begin{equation} \label{pp}
    \lVert P^{\nu}  \rVert_{L^{\infty} ([0,T] \times \Omega)} \le  e^{\Tilde{r} T}\,  \lVert P_{in}  \rVert_{L^{\infty} (\Omega)} .
\end{equation}

As a consequence, we infer from estimates (\ref{lim_sist_ineg})$_2$ and (\ref{lim_sist_ineg})$_3$ that
 \begin{equation}
\label{lim_sist_ineg2}
        \begin{cases}
            \dfrac{\partial M_s^{\nu}}{\partial t} \le \left(  \dfrac{\Tilde{\gamma}  }{\delta}  + \Gamma \right) \, M_h^{\nu} + d_2^{(1)} \Delta M_s^{\nu} ,\\[10pt]
            \dfrac{\partial M_h^{\nu}}{\partial t} \le \dfrac{\Tilde{\alpha} e^{\Tilde{r} T} \,  \lVert P_{in} \rVert_{L^{\infty} (\Omega)}  }{\delta}   M_s^{\nu} + d_2^{(2)} \Delta M_h^{\nu} .\\
        \end{cases}
    \end{equation}
    
Multiplying by $ (M_s^{\nu})^p$ (for some $p\ge 1$) estimate (\ref{lim_sist_ineg2})$_1$, by $ (M_h^{\nu})^p$ 
 estimate (\ref{lim_sist_ineg2})$_2$, adding up the result, integrating by parts and using \eqref{pp}, we see that
\begin{gather}
    \dfrac{1}{p+1}  \dfrac{d}{dt} \int_{\Omega} \left( [M_s^{\nu}]^{p+1} + [M_h^{\nu}]^{p+1} \right) = \int_{\Omega} \left( [M_s^{\nu}]^{p} \partial_t  M_s^{\nu} + [M_h^{\nu}]^{p} \partial_t M_h^{\nu}  \right) \\
    \leq  \left( \dfrac{\Tilde{\gamma}}{\delta} + \Gamma \right) \int_{\Omega}  M_h^{\nu} [M_s^{\nu}]^{p} +  \dfrac{\Tilde{\alpha}}{\delta} P^{\nu} \int_{\Omega}  M_s^{\nu} [M_h^{\nu}]^{p} - d_2^{(1)} p \int_{\Omega} ( M_s^{\nu})^{p-1} |\nabla M_s^{\nu}|^2  - d_2^{(2)} p \int_{\Omega} ( M_h^{\nu})^{p-1} |\nabla M_h^{\nu}|^2  \\
    \le C_T \int_{\Omega} \left( [M_s^{\nu}]^{p+1} + [M_h^{\nu}]^{p+1} \right),
\end{gather}
where
$C_T>0$ (here and in the rest of the proof) depends on $T$, $\delta$ (and other parameters), but neither on $\nu$ nor on 
$\varepsilon$.

Letting $p \to \infty$, we end up with
\begin{equation} 
\label{mm}
    ||M_s^{\nu}  ||_{L^{\infty} ([0,T] \times \Omega)} +  ||M_h^{\nu}  ||_{L^{\infty} ([0,T] \times \Omega)}  \le  C_T\, ( ||M_{s, in}  ||_{L^{\infty} (\Omega)} +  ||M_{h,in}  ||_{L^{\infty} (\Omega)}).
\end{equation}

As a consequence, we infer from estimates (\ref{lim_sist_ineg})$_4$ and  (\ref{lim_sist_ineg})$_5$ that

\begin{equation}
\label{lim_sist_ineg3}
\begin{cases}
    \dfrac{\partial T_s^{\nu}}{\partial t} \le \dfrac{\Tilde{\eta}}{\varepsilon} T_h^{\nu}+ \Tilde{s} \,T_s^{\nu} + d_3^{(1)} \Delta T_s^{\nu}, \\[10pt]
    \dfrac{\partial T_h^{\nu}}{\partial t}\le  \dfrac{\Tilde{\beta}}{\varepsilon} C_T\, ( ||M_{s,in}  ||_{L^{\infty} (\Omega)} +  ||M_{h,in} ||_{L^{\infty} (\Omega)})  T_s^{\nu} + \Tilde{s} \,T_h^{\nu} + d_3^{(2)} \Delta T_h^{\nu}  .
\end{cases}
\end{equation}

Estimating as previously, we end up with 
\begin{equation} 
\label{tt}
    ||T_s^{\nu}  ||_{L^{\infty} ([0,T] \times \Omega)} +  ||T_h^{\nu}  ||_{L^{\infty} ([0,T] \times \Omega)}  \le  C_{T}\, \left(1 + \frac1{\varepsilon}\right)\,  ( ||T_{s,in}  ||_{L^{\infty} (\Omega)} +  ||T_{h, in}  ||_{L^{\infty} (\Omega)}).
\end{equation}
Since
\begin{equation}
    \dfrac{ \Delta M_s^{\nu} }{(M_s^{\nu})^2}
    = - \Delta \left( \dfrac{1}{M_s^{\nu}} \right) + 2 \dfrac{|\nabla M_s^{\nu}|^2}{(M_s^{\nu})^3},
\end{equation}
recalling \eqref{lim_sist_appr}$_2$, we can next compute 
%
 $$   [\partial_t - d_2^{(1)} \Delta] \bigg( \frac1{M_s^{\nu}} \bigg) = - \frac{ [\partial_t - d_2^{(1)} \Delta] M_s^{\nu}}{(M_s^{\nu})^2} - 2d_2^{(1)} \frac{|\nabla M_s^{\nu}|^2}{(M_s^{\nu})^3} $$
$$    \le \frac1{(M_s^{\nu})^2} \, \bigg[    \dfrac{1}{\delta} \dfrac{\Tilde{\alpha} \frac{P^{\nu}}{1 + \nu [P^{\nu}]^2}  \frac{M_s^{\nu}}{1 + \nu [M_s^{\nu}]^2} }{1 + \Tilde{c}\, \sqrt{ \nu^2 + \left(T_s^{\nu}+T_h^{\nu}\right)^2 } } + \Tilde{\mu} \frac{M_s^{\nu}}{1 + \nu [M_s^{\nu}]^2} + \Tilde{\beta} \frac{M_s^{\nu}}{1 + \nu [M_s^{\nu}]^2} \, \frac{T_s^{\nu}}{1 + \nu [T_s^{\nu}]^2} \bigg] $$
   \begin{equation}
  \le \frac{C_{T}}{M_s^{\nu}}\, \left(1 + \frac1{\varepsilon}\right),
\end{equation}
where the last inequality uses estimates \eqref{pp}, \eqref{mm} and \eqref{tt}.

Then, thanks to the maximum principle, 
\begin{equation}
    \bigg|\bigg| \frac{1}{M_s^{\nu}}  \bigg|\bigg|_{L^{\infty} ([0,T] \times \Omega)}  \le  C_{T}\, \left(1 + \frac1{\varepsilon}\right)\, \bigg|\bigg| \frac1{M_{s,in}}  \bigg|\bigg|_{L^{\infty} (\Omega)} ,
\end{equation}
which can be rewritten (for $t,x \in [0,T] \times \Omega$)
\begin{equation} \label{dd1}
    M_s^{\nu}(t,x) \ge  \frac1{C_{T}\, \left(1 + \frac1{\varepsilon}\right)}>0 .
\end{equation}

A similar computation leads to
\begin{equation} \label{dd2}
    M_h^{\nu}(t,x) \ge \frac1{C_{T}\, \left(1 + \frac1{\varepsilon}\right)}>0 .
\end{equation}

As a consequence of estimates \eqref{pp}, \eqref{mm}, \eqref{tt}, and \eqref{dd1}, \eqref{dd2}, we get from \eqref{lim_sist_appr} the estimates

\begin{equation} \label{hpp}
 || [\partial_t - d_1 \Delta] P^{\nu} ||_{L^{\infty} ([0,T] \times \Omega)}  \le C_T, 
\end{equation}

\begin{equation} 
\label{hmm}
 || [\partial_t - d_2^{(1)} \Delta] M_s^{\nu} ||_{L^{\infty} ([0,T] \times \Omega)}  \leq C_{T}\, \left(1 + \frac1{\varepsilon}\right),  \qquad || [\partial_t - d_2^{(2)} \Delta] M_h^{\nu} ||_{L^{\infty} ([0,T] \times \Omega)}  \le C_{T}\, \left(1 + \frac1{\varepsilon}\right), 
\end{equation}

\begin{equation}  \label{htt}
|| [\partial_t - d_3^{(1)} \Delta] T_s^{\nu} ||_{L^{\infty} ([0,T] \times \Omega)}  \le C_{T}\, \left(1 + \frac1{\varepsilon}\right)^4,  \qquad || [\partial_t - d_3^{(2)} \Delta] T_h^{\nu} ||_{L^{\infty} ([0,T] \times \Omega)}  \le C_{T}\, \left(1 + \frac1{\varepsilon}\right)^4.
\end{equation}
Notice that the bounds \eqref{dd1} and \eqref{dd2} are used in estimate \eqref{htt}.

Using maximal regularity for the heat equation and the bounds \eqref{hpp}, \eqref{hmm}, \eqref{htt}, we see that the sequences $P^{\nu}$, $M_s^{\nu}$, $M_h^{\nu}$, $T_s^{\nu} $,  $T_h^{\nu}$ converge when $\nu \to 0$  (up to a subsequence) a.e. to 
some $P$, $M_s$, $M_h$, $T_s$,  $T_h$ (all those quantities in fact depend on $\varepsilon$), which are all in $L^{\infty} ([0,T] \times \Omega)$, and such that their derivative with respect to time, and their second derivatives with respect to space are in  $L^{p} ([0,T] \times \Omega)$ for all $p\in [1, \infty[$ and $T >0$. They also satisfy estimate 
\begin{equation} \label{newed}
    M_s^{\nu}(t,x) \ge \frac1{C_{T}\, \left(1 + \frac1{\varepsilon}\right)} >0 , \qquad
    M_h^{\nu}(t,x) \ge \frac1{C_{T}\, \left(1 + \frac1{\varepsilon}\right)} >0.
\end{equation}

 Passing to the limit in system \eqref{lim_sist_appr}, with Neumann boundary condition [and initial data (\ref{titi}) which anyway do not depend on $\nu$], we see that  $P$, $M_s$, $M_h$, $T_s$,  $T_h$ (which in fact depend on $\varepsilon$),
satisfy in the strong sense system (\ref{lim_sist_rig}) -- (\ref{titi}). Indeed all terms (in the system) without derivatives converge a.e. (towards the right quantity), while the terms with derivatives (including the trace of the gradients appearing in the  Neumann boundary condition) converge weakly in  $L^{p} ([0,T] \times \Omega)$ for all $p\in [1, \infty[$ and $T>0$ (once again, to the right quantity). 

Bootstrapping, one can show that  $P$, $M_s$, $M_h$, $T_s$,  $T_h$ are in fact smooth (remember estimate \eqref{newed}). Uniqueness also holds when one restricts oneself to the class of smooth solutions such that $P$, $T_s$,  $T_h$ are nonnegative, and such that $M_s$, $M_h$ are strictly positive.
\medskip

{\bf{Second step}}: $\varepsilon$-independent a priori estimates
\medskip

We now reintroduce the parameter $\varepsilon$ in the notations (since we will let $\varepsilon$ go  to $0$), and consider the unique smooth solution $P^\varepsilon$, $M_s^\varepsilon$, $M_h^\varepsilon$, $T_s^\varepsilon$,  $T_h^\varepsilon$ of   system (\ref{lim_sist_rig}) -- (\ref{titi}) obtained in the first step.   We recall that $P^\varepsilon \ge 0$, $T_s^\varepsilon \ge 0$,  $T_h^\varepsilon \ge 0$, and  $M_s^\varepsilon >0$, $M_h^\varepsilon >0$.

We first observe that the constants are not depending on $\varepsilon$ in estimates \eqref{pp} and \eqref{mm}, so
that
$$ ||P^{\varepsilon}  ||_{L^{\infty} ([0,T] \times \Omega)} \le  C_T\,  ||P_{in} ||_{L^{\infty} (\Omega)} , $$
$$ ||M_s^{\varepsilon}  ||_{L^{\infty} ([0,T] \times \Omega)} +  || M_h^{\varepsilon} ||_{L^{\infty} ([0,T] \times \Omega)}  \le  C_T\, ( ||M_{s,in} ||_{L^{\infty} (\Omega)} +  || M_{h,in}  ||_{L^{\infty} (\Omega)}).$$
However, it is not possible to obtain in a similar way an $\varepsilon$-independent $L^{\infty}$ estimate for $T_s^{\varepsilon}$ and $T_h^{\varepsilon}$, because of the factor $\frac1{\varepsilon}$  in the right-hand side of the last two equations of \eqref{lim_sist_rig}.
\medskip

In order to obtain an $\varepsilon$-independent  estimate for $T_h^{\varepsilon}$ and $T_s^{\varepsilon}$ (in a weaker norm),  we use the (refined version of the) so called duality lemma (see \cite{CDF} or \cite{BDF}), and the estimate
\begin{equation}
    \partial_t (T_s^{\varepsilon} + T_h^{\varepsilon}) - \Delta ( d_3^{(1)} T_s^{\varepsilon} + d_3^{(2)} T_h^{\varepsilon} ) \le  \Tilde{s} \,( T_s^{\varepsilon} + T_h^{\varepsilon} ),
\end{equation}
to show that for some $\zeta>0$, 
\begin{equation} \label{ze}
    \lVert T_s^{\varepsilon}  \rVert_{L^{2+\zeta} ([0,T] \times \Omega)} + \lVert T_h^{\varepsilon}  \rVert_{L^{2+\zeta} ([0,T] \times \Omega)} \le C_T .
\end{equation}
Using estimate \eqref{ze} and observing that 
\begin{equation}
 | [\partial_t - d_1 \Delta]  P^{\varepsilon}|  \le C_T , \qquad 
    | [\partial_t - d_2^{(1)} \Delta]  M_s^{\varepsilon}|  \le C_T\,(1 + T_s^{\varepsilon}) , \qquad  | [\partial_t - d_2^{(2)} \Delta]  M_h^{\varepsilon}|  \le C_T\,(1 + T_s^{\varepsilon}) ,
\end{equation}
 the maximal regularity property of the heat equation ensures that
 (for any $q \in [1,\infty[$ and $T>0$)
\begin{equation}
    || \partial_t  P^{\varepsilon} ||_{L^{q} ([0,T] \times \Omega)} + \sum_{i,j = 1,..,N} || \partial_{x_i x_j}  P^{\varepsilon} ||_{L^{q} ([0,T] \times \Omega)}  \le C_T ,
\end{equation}
and
\begin{equation}
    || \partial_t  M_s^{\varepsilon} ||_{L^{2+\zeta} ([0,T] \times \Omega)} + \sum_{i,j = 1,..,N} || \partial_{x_i x_j}  M_s^{\varepsilon} ||_{L^{2+\zeta} ([0,T] \times \Omega)}  \le C_T ,
\end{equation}
%
\begin{equation}
     || \partial_t  M_h^{\varepsilon} ||_{L^{2+\zeta} ([0,T] \times \Omega)} + \sum_{i,j = 1,..,N} || \partial_{x_i x_j}  M_h^{\varepsilon} ||_{L^{2+\zeta} ([0,T] \times \Omega)}  \le C_T .
\end{equation}
\medskip

We conclude that up to extraction of a subsequence, $P^{\varepsilon}$, $M_s^{\varepsilon}$ and $M_h^{\varepsilon}$ converge a.e. towards $P$, $M_s$ and $M_h$, where $P$, $M_s$ and $M_h$ lie in $L^{\infty} ([0,T] \times \Omega)$ for all $T>0$.
We also know thanks to \eqref{ze} 
that up to a subsequence, $T_s^{\varepsilon}$ and $T_h^{\varepsilon}$ converge weakly in $L^{2+\zeta} ([0,T] \times \Omega)$
towards  $T_s$ and $T_h$, where  $T_s$ and $T_h$ lie in $L^{2+\zeta} ([0,T] \times \Omega)$ (for all $T>0$ and some $\zeta>0$). 
\medskip

At this point, we observe that
 $$   [\partial_t - d_2^{(1)} \Delta] \ln  M_s^{\varepsilon}  = \frac{ [\partial_t - d_2^{(1)} \Delta]  M_s^{\varepsilon} }{M_s^{\varepsilon} }  + d_2^{(1)} \frac{|\nabla M_s^{\varepsilon}|^2}{ (M_s^{\varepsilon})^2} $$
\begin{equation}
    \geq - \frac{ \Tilde{\alpha} }{\delta}  \, \frac{P^{\varepsilon}}{1 + \Tilde{c}\, \left(T_s^{\varepsilon}+T_h^{\varepsilon}\right)}  - \Tilde{\mu} - \Tilde{\beta} \, T_s^{\varepsilon} \geq - \Tilde{\beta} \,  T_s^{\varepsilon} - C_T . 
\end{equation}
In the same way,  
 \begin{equation}
   [\partial_t - d_2^{(1)} \Delta] \ln  M_h^{\varepsilon} 
    \geq - \frac{ \Tilde{\gamma} }{\delta}  - \Tilde{\mu} - \Tilde{\beta} \, T_s^{\varepsilon} \geq - \Tilde{\beta} \,  T_s^{\varepsilon} - C_T . 
\end{equation}
%
\medskip

The semigroup property of the heat equation (that is, when $[\partial_t - d \Delta] u \in L^{q} ([0,T] \times \Omega)$ for some $q>0$, then 
$u \in L^{r} ([0,T] \times \Omega)$ with $\dfrac1r > \dfrac1q + \dfrac{N}{N+2} - 1$; or $r=\infty$ when $\dfrac1q + \dfrac{N}{N+2} < 1$)  implies that $ \ln M_s^{\varepsilon} \ge - C_T$, so that 
(for $t,x \in [0,T] \times \Omega$)
\begin{equation}
    M_s^{\varepsilon}(t,x) \ge e^{- C_T} >0 , \qquad 
    M_h^{\varepsilon}(t,x) \ge e^{- C_T} >0 .
\end{equation}
Indeed, when $q= 2 + \zeta > 2$ and $N=1$ or $2$, then $\frac1q + \frac{N}{N+2} < 1$, so that 
$$ \ln M_s^\varepsilon \ge - C_T, \qquad \ln M_h^\varepsilon \ge - C_T. $$

\medskip

We  compute then, for any $\alpha \in ]0,1[$, denoting $M^\varepsilon :=  M_s^\varepsilon + M_h^\varepsilon$ and $T^\varepsilon_* :=  T_s^\varepsilon + T_h^\varepsilon$,

$$ \frac{d}{dt} \bigg\{ \int_{\Omega} \bigg[ \tilde{\eta}^{\alpha} \, \frac{(T_h^{\varepsilon})^{1+\alpha}}{1+\alpha}  + \tilde{\beta}^{\alpha}
\, (M^{\varepsilon})^{\alpha}\, \frac{(T_s^{\varepsilon})^{1+\alpha}}{1+\alpha} \bigg] \, \bigg\}$$
$$ = \int_{\Omega} \bigg[ \tilde{\eta}^{\alpha}\, (T_h^{\varepsilon})^{\alpha}\, \partial_t T_h^{\varepsilon} +  \tilde{\beta}^{\alpha} \,  (M^{\varepsilon})^{\alpha}\,  (T_s^{\varepsilon})^{{\alpha}} \, \partial_t T_s^{\varepsilon} + \frac{\alpha}{1+\alpha} \, 
\tilde{\beta}^{\alpha}\,( T_s^{\varepsilon})^{1+\alpha}\, (M^{\varepsilon})^{\alpha - 1}\,
\partial_t M^{\varepsilon} \,\bigg] $$

$$ = \int_{\Omega} \bigg[ \tilde{\eta}^{\alpha}\, (T_h^{\varepsilon})^{\alpha}\, d_3^{(2)}\, \Delta T_h^{\varepsilon} + \tilde{s}\, \tilde{\eta}^{\alpha}\, (T_h^{\varepsilon})^{1+\alpha}\, \left(1 - \frac{T^{\varepsilon}_*}{ \tilde{m} \, M^{\varepsilon} } \right) $$
$$ - \frac1{\varepsilon}\,  \tilde{\eta}^{\alpha}\, (T_h^{\varepsilon})^{\alpha}\, (\tilde{\eta}\, T_h^{\varepsilon} -  \tilde{\beta}  \, M^{\varepsilon}  \, T_s^{\varepsilon} )
+  \tilde{\beta}^{\alpha}\,  d_3^{(1)}\,  (M^\varepsilon)^{\alpha} \,  (T_s^{\varepsilon})^{{\alpha}} \Delta T_s^{\varepsilon}
 + \tilde{s} \,{\tilde{\beta}^{\alpha}}\, (M^\varepsilon)^{\alpha} \, (T_s^{\varepsilon})^{1+\alpha}\, \left(1 - \frac{T^{\varepsilon}_*}{ \tilde{m} \, M^{\varepsilon}  } \right)$$
$$ +  \frac1{\varepsilon}\,  \tilde{\beta}^{\alpha}  (M^{\varepsilon})^{\alpha}\,  (T_s^{\varepsilon})^{{\alpha}}
  \, ( \tilde{\eta}\, T_h^{\varepsilon} -  \tilde{\beta}  \, M^{\varepsilon}  \, T_s^{\varepsilon} ) 
+ \frac{\alpha}{1+\alpha}\, \tilde{\beta}^{\alpha}\, (T_s^{\varepsilon})^{1+\alpha}\, (M^{\varepsilon})^{\alpha - 1} \,
 \partial_t M^{\varepsilon}\, \bigg] $$
 
$$ = - \frac1{\varepsilon} \int_{\Omega} (\tilde{\eta}\, T_h^{\varepsilon} -  \tilde{\beta}  \, M^{\varepsilon}  \, T_s^{\varepsilon} ) \,
((\tilde{\eta}\, T_h^{\varepsilon})^{\alpha} -  (\tilde{\beta}  \, M^{\varepsilon}  \, T_s^{\varepsilon} )^{\alpha} ) $$
$$ + \tilde{s} \, \int_{\Omega}  [  \tilde{\eta}^{\alpha}\, (T_h^{\varepsilon})^{1+\alpha} +   {\tilde{\beta}^{\alpha}}\, (M^{\varepsilon})^{\alpha} \, (T_s^{\varepsilon})^{1+\alpha}]\, \left(1 - \frac{T_{*}^{\varepsilon}}{ \tilde{m} \, M^{\varepsilon}  } \right) $$
$$ - \tilde{\eta}^{\alpha}\, d_3^{(2)}\,  \alpha \int_{\Omega}  (T_h^{\varepsilon})^{\alpha -1} |\nabla T_h^{\varepsilon}|^2  - \tilde{\beta}^{\alpha}\,  d_3^{(1)}\, \alpha \int_{\Omega}  (M^{\varepsilon})^{\alpha}\,  (T_s^{\varepsilon})^{\alpha -1} |\nabla T_s^{\varepsilon}|^2  $$
$$ + \frac{\alpha}{1+\alpha}\, \tilde{\beta}^{\alpha}\, \int_{\Omega} (T_s^{\varepsilon})^{1+\alpha}\,  \bigg\{ (M^{\varepsilon})^{\alpha - 1} \,
 \partial_t M^{\varepsilon} + d_3^{(1)}\,  \nabla\cdot (  \, (M^{\varepsilon})^{\alpha - 1} \nabla M^{\varepsilon}  ) \bigg\} .$$

Integrating with respect to time between $0$ and $T$ leads to the following estimate:
$$ \int_{\Omega} \bigg( \tilde{\eta}^{\alpha} \, \frac{(T_h^{\varepsilon})^{1+\alpha}}{1+\alpha}  + \tilde{\beta}^{\alpha}
\, (M^{\varepsilon})^{\alpha}\, \frac{(T_s^{\varepsilon})^{1+\alpha}}{1+\alpha} \bigg) \, (T)  $$
$$ + \frac1{\varepsilon} \int_0^T \int_{\Omega} (\tilde{\eta}\, T_h^{\varepsilon} -  \tilde{\beta}  \, M^{\varepsilon}  \, T_s^{\varepsilon} ) \,
((\tilde{\eta}\, T_h^{\varepsilon})^{\alpha} -  (\tilde{\beta}  \, M^{\varepsilon}  \, T_s^{\varepsilon} )^{\alpha} ) $$
$$ +\tilde{\eta}^{\alpha}\, d_3^{(2)}\,  \alpha \int_0^T \int_{\Omega}   (T_h^{\varepsilon})^{\alpha -1} |\nabla T_h^{\varepsilon}|^2  + \tilde{\beta}^{\alpha}\,  d_3^{(1)}\, \alpha \int_0^T\int_{\Omega}  (M^{\varepsilon})^{\alpha}\,  (T_s^{\varepsilon})^{\alpha -1} |\nabla T_s^{\varepsilon}|^2  $$
$$ \le \tilde{s} \, \int_0^T\int_{\Omega}   [  \tilde{\eta}^{\alpha}\, (T_h^{\varepsilon})^{1+\alpha} +   {\tilde{\beta}^{\alpha}}\, (M^{\varepsilon})^{\alpha} \, (T_s^{\varepsilon})^{1+\alpha}] $$
$$ + \frac{\alpha}{1+\alpha}\, \tilde{\beta}^{\alpha}\,\int_0^T \int_{\Omega} (T_s^{\varepsilon})^{1+\alpha}\, (M^{\varepsilon})^{\alpha - 1} \,
 \bigg[ |\partial_t M^{\varepsilon}| + d_3^{(1)}\, \Big(|\Delta M^{\varepsilon}| + (\alpha - 1)  \, \frac{|\nabla  M^{\varepsilon}|^2}{M^{\varepsilon}} \Big) \bigg] $$
$$ + \int_{\Omega} \bigg( \tilde{\eta}^{\alpha} \, \frac{(T_{h,in})^{1+\alpha}}{1+\alpha}  + \tilde{\beta}^{\alpha}
\, (M_{s,in} + M_{h,in})^{\alpha}\, \frac{(T_{s,in})^{1+\alpha}}{1+\alpha} \bigg) . $$

The first term in the right-hand side of the estimate above is bounded (uniformly in $\varepsilon$) since
$ M^{\varepsilon}$ is bounded (uniformly in $\varepsilon$) in
$L^{\infty}([0,T]\times\Omega)$ for all $T>0$, and since $T_h^{\varepsilon},  T_s^{\varepsilon}$ are bounded in $L^{2+\zeta}([0,T] \times\Omega)$ for all $T>0$, and some $\zeta>0$. 
\par
The last term of this  right-hand side is also finite thanks to the assumptions made on the initial data.
\par 
Remembering finally that $\partial_t M^{\varepsilon}$ and  $\partial_{x_ix_j} M^{\varepsilon}$ are bounded in $L^{2+\zeta}([0,T] \times\Omega)$ (for some $\zeta>0$, and all $T>0$, $i,j \in \{1,..,N\}$), and that $M^{\varepsilon} \ge e^{- C_T }>0$, we see that when $\alpha>0$ is small enough, the second term is also bounded (uniformly in $\varepsilon$). Note that $\alpha - 1 <0$  so that $(\alpha - 1)\, \frac{|\nabla  M^{\varepsilon}|^2}{M^{\varepsilon}} \le 0$.
\medskip

Still assuming that $\alpha>0$ is small enough, we get therefore the following bounds: 
\begin{equation}\label{coer}
 \int_0^T \int_{\Omega} (\tilde{\eta}\, T_h^{\varepsilon} -  \tilde{\beta}  \, M^{\varepsilon}  \, T_s^{\varepsilon} ) \,
((\tilde{\eta}\, T_h^{\varepsilon})^{\alpha} -  (\tilde{\beta}  \, M^{\varepsilon}  \, T_s^{\varepsilon} )^{\alpha} ) 
 \le C_T\, \varepsilon , 
\end{equation}
and 
 $$ 
 \int_0^T \int_{\Omega}   (T_h^{\varepsilon})^{\alpha -1} |\nabla T_h^{\varepsilon}|^2  \le C_T, \qquad   \int_0^T\int_{\Omega}  M^{\alpha}_{\varepsilon}\,  (T_s^{\varepsilon})^{\alpha -1} |\nabla T_s^{\varepsilon}|^2 
\le C_T. $$ 
\par 
Then, using Cauchy-Schwartz inequality, we get the estimate
\begin{equation}\label{w11}
\bigg(\int_0^T \int_{\Omega}  |\nabla T_h^{\varepsilon}| \,  \bigg)^2 
\le  C_T \int_0^T \int_{\Omega}   (T_h^{\varepsilon})^{1 - \alpha} \le C_T,
\end{equation}
 and in the same way, 
\begin{equation}\label{w11bis}
\bigg(\int_0^T \int_{\Omega} |\nabla T_s^{\varepsilon} | \,  \bigg)^2 
\le  \int_0^T \int_{\Omega}    M^{-\alpha}_{\varepsilon}\,  (T_s^{\varepsilon})^{1 - \alpha}    \le C_T.
\end{equation}
\medskip 

Moreover,  
$\partial_t T^{\varepsilon}_*$ is bounded  in $L^2([0,T] ; H^{-2}(\Omega)) +  L^{1 + \zeta/2}([0,T] \times\Omega)$ since $M^{\varepsilon} \geq e^{- C_T} >0$,
so that thanks to estimates (\ref{w11}) and (\ref{w11bis}) and  Aubin-Lions lemma,
$T^{\varepsilon}_*$ converges (up to extraction of a subsequence) a.e. to $T_*$ on $[0,T] \times\Omega$.
\medskip

Then, using the elementary inequality (for $\alpha\in ]0,1[$, and a constant $C_{\alpha}$  which may depend on $\alpha$)
$$  (x^{(1+\alpha)/2} - y^{(1+\alpha)/2})^2 \le  C_{\alpha} \,  (x -y)\, (x^{\alpha} - y^{\alpha}), $$
 estimate (\ref{coer}) leads to the bound: 
$$ \int_0^T \int_{\Omega} \bigg| \bigg[ \tilde{\eta}\, T_h^{\varepsilon} \bigg]^{\frac{1+\alpha}2} - \bigg[   \tilde{\beta}  \, M^{\varepsilon}  \, T_s^{\varepsilon}  \bigg]^{\frac{1+\alpha}2} \bigg|^2\, 
 dx dt  \le C_{\alpha}\, C_T \, \varepsilon . $$
We now introduce a second elementary inequality (which holds for $\alpha>0$ small enough, and a constant $C'_{\alpha}$  which may depend on $\alpha$)
$$|x - y| \le C'_{\alpha}\, |x^{(1+\alpha)/2} - y^{(1+\alpha)/2}|\, (x^{(1-\alpha)/2} + y^{(1-\alpha)/2}). $$
 Then
$$ \int_0^T \int_{\Omega} | \tilde{\eta}\, T_h^{\varepsilon} -  \tilde{\beta}  \, M^{\varepsilon}  \, T_s^{\varepsilon} | $$ 
 $$ \le C'_{\alpha} \int_0^T \int_{\Omega} \bigg|   \bigg[ \tilde{\eta}\, T_h^{\varepsilon} \bigg]^{\frac{1+\alpha}2} - \bigg[   \tilde{\beta}  \, M^{\varepsilon}  \, T_s^{\varepsilon}  \bigg]^{\frac{1+\alpha}2}   \bigg|\, \bigg(    \bigg[ \tilde{\eta}\, T_h^{\varepsilon} \bigg]^{\frac{1+\alpha}2} + \bigg[   \tilde{\beta}  \, M^{\varepsilon}  \, T_s^{\varepsilon}  \bigg]^{\frac{1+\alpha}2}    \bigg)
   $$
 $$ \leq  C'_{\alpha} \,  (C_{\alpha} {\varepsilon})^{1/2} \, C_T . $$
 As a consequence, 
$\tilde{\eta}\, T_h^{\varepsilon} -  \tilde{\beta}  \, M^{\varepsilon}  \, T_s^{\varepsilon}  $ converges (up to extraction of a subsequence)
 strongly in $L^1([0,T] \times\Omega)$ and a.e. to $0$.
 Recalling that $M^\varepsilon$ converges a.e. towards $M$,  and that $ T_h^{\varepsilon},  \, T_s^{\varepsilon}  $
converge weakly in $L^{2+\zeta}([0,T] \times\Omega)$ towards $T_h, T_s$ respectively, we see that
$\tilde{\eta}\, T_h^{\varepsilon} -  \tilde{\beta}  \, M^{\varepsilon}  \, T_s^{\varepsilon}  $ converges
weakly in $L^1([0,T] \times\Omega)$ towards
$ \tilde{\eta}\, T_h -  \tilde{\beta}  \, M \, T_s$, so that  $ \tilde{\eta}\, T_h =  \tilde{\beta}  \, M \, T_s$. 
\medskip

Remembering moreover that $T_{*}^{\varepsilon}$ converges a.e. to $T_*$, we see that 
$ \tilde{\eta}\, T_h^{\varepsilon}+ \tilde{\eta}\, T_s^{\varepsilon}$ converges a.e. to $\tilde{\eta}\, T_*$,  and (since $\tilde{\eta}\, T_h^{\varepsilon} -  \tilde{\beta}  \, M^{\varepsilon}  \, T_s^{\varepsilon}  $ converges to $0$ a.e.)
 $ (\tilde{\beta}  \, M^{\varepsilon} + \tilde{\eta}) \, T_s^{\varepsilon}  $  converges to $ \tilde{\eta}\, T_* $,
so that $T_s^{\varepsilon}$ converges to $\frac{ \tilde{\eta} T_*}{ \tilde{\eta} + \tilde{\beta}  \, M }$ a.e, and 
$T_h^{\varepsilon}$ converges to $\frac{ \tilde{\beta}  \, M \, T_*}{ \tilde{\eta} + \tilde{\beta}  \, M }$ a.e.

{\bf{Third step}}: Passage to the limit
\medskip

Taking up the first equation of system \eqref{lim_sist_rig},
the very-weak formulation of the first equation 
is given by: \\
$\forall \varphi \in C_c^2([0, \infty) \times \Bar{\Omega})$ such that $\nabla_x \varphi \cdot \textbf{n} |_{\partial \Omega} = 0,$    
    \begin{equation}        
        - \underbrace{\int_0^{\infty} \int_{\Omega} P^{\varepsilon} \partial_t \varphi}_{\fbox{1}} - \int_{\Omega} \varphi(0, \cdot) P_{in} - \underbrace{d_1 \int_0^{\infty} \int_{\Omega} P^{\varepsilon} \Delta \varphi}_{\fbox{2}} = \underbrace{\int_0^{\infty} \int_{\Omega}  \Tilde{r} \left(1 - \dfrac{P^{\varepsilon}}{K} \right) P^{\varepsilon} \varphi}_{\fbox{3}} -  \underbrace{\int_0^{\infty} \int_{\Omega} \dfrac{\Tilde{\alpha} P^{\varepsilon} M_s^{\varepsilon}}{1 + \Tilde{c}\, \left(T_s^{\varepsilon} + T_h^{\varepsilon} \right)} \varphi}_{\fbox{4}}. 
    \end{equation}
    We can easily pass to the limit $\varepsilon \to 0$ in the terms ${\fbox{1}}, \, {\fbox{2}}, \, {\fbox{3}} \, {\fbox{4}}$ using the Lebesgue's dominated convergence theorem. 
    Indeed, since $P^{\varepsilon}$, $M_s^{\varepsilon}$,  $T_s^{\varepsilon}$, $T_h^{\varepsilon}$ converge a.e. towards $P$, $M_s$,  $T_s$, $T_h$ and previous estimates (see second step), we know that 
    $$\lVert P^{\varepsilon}\partial_t \varphi\rVert_{L^{\infty}([0,T] \times \Omega)}  +  \lVert P^{\varepsilon}\Delta \varphi\rVert_{L^{\infty}([0,T] \times \Omega)} \leq C_T \, \Big(  \lVert \partial_t \varphi \rVert_{L^{\infty}([0,T] \times \Omega)} +
  \lVert \Delta \varphi \rVert_{L^{\infty}([0,T] \times \Omega)} \Big)\, 
,$$
    and
     $$\bigg|\bigg| \dfrac{\Tilde{\alpha} P^{\varepsilon} M_s^{\varepsilon}}{1 + \Tilde{c}\, \left(T_s^{\varepsilon} + T_h^{\varepsilon} \right)} \varphi \bigg|\bigg|_{L^{\infty}([0,T] \times \Omega)}  \leq  C_T \lVert \varphi \rVert_{L^{\infty}([0,T] \times \Omega)} 
, $$
    so that we obtain \eqref{1}.
    
\vspace{0.3cm}

    In the same way, the very-weak formulation of the second equation of system \eqref{lim_sist_rig} is given by: \\
    $\forall \psi \in C_c^2([0, \infty) \times \Bar{\Omega})$ such that $\nabla_x \psi \cdot \textbf{n} |_{\partial \Omega} = 0,$ 
    \begin{equation}
        \begin{aligned}
            - \int_0^{\infty}& \int_{\Omega} M_s^{\varepsilon} \partial_t \psi - \int_{\Omega} \psi(0, \cdot) M_{s,in} - \int_0^{\infty} \int_{\Omega} d_2^{(1)} M_s^{\varepsilon} \Delta \psi  = \\
            &= \int_0^{\infty} \int_{\Omega}  \left[ \Gamma M_h^{\varepsilon} - \Tilde{\mu} M_s^{\varepsilon} \right] \psi - \underbrace{\int_0^{\infty} \int_{\Omega} \Tilde{\beta} M_s^{\varepsilon} T_s^{\varepsilon} \psi}_{\fbox{5}} + \int_0^{\infty} \int_{\Omega} \dfrac{1}{\delta} \left[  - \dfrac{\Tilde{\alpha} P^{\varepsilon} M_s^{\varepsilon}}{1 + \Tilde{c}\, \left(T_s^{\varepsilon}+T_h^{\varepsilon} \right)} + \Tilde{\gamma} M_h^{\varepsilon} \right] \psi  . \\[10pt]
        \end{aligned}
    \end{equation}
    We can easily pass to the limit $\varepsilon \to 0$ in almost all the terms using  Lebesgue's dominated convergence theorem as in the equation for $P^\varepsilon$. 
The only slightly different term is ${\fbox{5}}$. Since $T_s^{\varepsilon}$ converges strongly in $L^{2}([0,T]\times \Omega)$ 	and $M_s^{\varepsilon}$ converges to $M_s$ strongly in $L^{2}([0,T]\times \Omega)$,
then $M_s^{\varepsilon} T_s^{\varepsilon}$ converge weakly to $M_s T_s$ in $L^1$. This enables to pass to the limit in ${\fbox{5}}$, and leads to eq. \eqref{2}.
\medskip

Eq. \eqref{3} can be obtained in the exact same way.
\medskip

     Finally, the very-weak formulation of the sum of the fourth and fifth equations of system \eqref{lim_sist_rig} is given by: \\
    $\forall \xi \in C_c^2([0, \infty) \times \Bar{\Omega})$ such that $\nabla_x \xi \cdot \textbf{n} |_{\partial \Omega} = 0,$ 
    \begin{equation}
        \begin{aligned}
            - \int_0^{\infty} \int_{\Omega} (T_s^{\varepsilon} + T_h^{\varepsilon}) \partial_t \xi &- \int_{\Omega} \xi(0, \cdot) (T_{s, in} + T_{h,in}) - \int_0^{\infty} \int_{\Omega} \left( d_3^{(1)}  T_s^{\varepsilon} +  d_3^{(2)}  T_h^{\varepsilon} \right) \Delta \xi = \\
            &= \int_0^{\infty} \int_{\Omega} \Tilde{s} \,(T_s^{\varepsilon} + T_h^{\varepsilon}) \, \xi - 
            \underbrace{\int_0^{\infty} \int_{\Omega} \Tilde{s} \,(T_s^{\varepsilon} + T_h^{\varepsilon}) \left( \dfrac{T_s^{\varepsilon} + T_h^{\varepsilon} }{ \Tilde{m} (M_h^{\varepsilon} + M_s^{\varepsilon})} \right)  \xi}_{\fbox{6}}.
        \end{aligned}
    \end{equation}
    For all terms, except for $\fbox{6}$, one can proceed as previously. Concerning term $\fbox{6}$, we know that 
$\dfrac{(T_s^{\varepsilon} + T_h^{\varepsilon})^2}{M_h^{\varepsilon} + M_s^{\varepsilon}}$  converges a.e. towards 
    $\dfrac{(T_s + T_h)^2}{M_h + M_s}$ (remember that $M_h^{\varepsilon} + M_s^{\varepsilon} \ge e^{- C_T}>0$). 
Moreover, $\dfrac{(T_s^{\varepsilon} + T_h^{\varepsilon})^2}{M_h^{\varepsilon} + M_s^{\varepsilon}}$ is bounded in
$L^{1+ \zeta/2}([0,T] \times \Omega)$ for some $\zeta >0$, and all $T>0$. 
As a consequence, 
$$ \int_0^{\infty} \int_{\Omega} \Tilde{s} \,(T_s^{\varepsilon} + T_h^{\varepsilon}) \left( \dfrac{T_s^{\varepsilon} + T_h^{\varepsilon} }{ \Tilde{m} (M_h^{\varepsilon} + M_s^{\varepsilon})} \right)  \xi
    \to_{\varepsilon \to 0} 
\int_0^{\infty} \int_{\Omega} \Tilde{s} \,(T_s + T_h) \left( \dfrac{T_s + T_h}{ \Tilde{m} (M_h + M_s)} \right)  \xi .$$
All in all, we end up with eq. \eqref{tteq}. We recall that we ended up at the end of step 2 with the constraint \eqref{exf}, so that $P$, $M_s$, $M_h$, $T_s$,  $T_h$ solves as announced the very weak formulation of \eqref{senza eps} -- \eqref{titi2} (or equivalently \eqref{simpform} -- \eqref{titi3}), that is  \eqref{exf} -- \eqref{tteq}.

\end{proof}

\bibliographystyle{amsplain}
\bibliography{references}

\providecommand{\bysame}{\leavevmode\hbox to3em{\hrulefill}\thinspace}
\providecommand{\MR}{\relax\ifhmode\unskip\space\fi MR }
\providecommand{\MRhref}[2]{%
  \href{http://www.ams.org/mathscinet-getitem?mr=#1}{#2}
}
\providecommand{\href}[2]{#2}
\begin{thebibliography}{10}

\bibitem{AG2000}
P.~A. Abrams and L.~R. Ginzburg, \emph{The nature of predation: prey dependent,
  ratio dependent or neither?}, Trends in Ecology \& Evolution \textbf{15}
  (2000), no.~8, 337--341.

\bibitem{Beddington1975}
J.~R. Beddington, \emph{Mutual interference between parasites or predators and
  its effect on searching efficiency}, The Journal of Animal Ecology (1975),
  331--340.

\bibitem{BDF}
M.~Breden, L.~Desvillettes, and K.~Fellner, \emph{Smoothness of moments of the
  solutions of discrete coagulation equations with diffusion}, Monatshefte
  f{\"u}r Mathematik \textbf{183} (2017), 437--463.

\bibitem{CdLFLM2023}
F.~Capone, R.~De~Luca, L.~Fiorentino, V.~Luongo, and G.~Massa, \emph{Turing
  instability for a {L}eslie--{G}ower model}, Ricerche di Matematica (2023),
  1--18.

\bibitem{CDF}
J.A. Cañizo, L.~Desvillettes, and K.~Fellner, \emph{Improved duality estimates
  and applications to reaction-diffusion equations}, Communications in Partial
  Differential Equations 39 (6) (2014), 1185--1204.

\bibitem{CC}
F.~Chen and L.~Chen, \emph{Global stability of a leslie–gower predator–prey
  model with feedback controls}, Applied Mathematics Letters \textbf{22}
  (2009), no.~9, 1330--1334.

\bibitem{CCX}
F.~Chen, L.~Chen, and X.~Xie, \emph{On a {L}eslie–{G}ower predator–prey
  model incorporating a prey refuge}, Nonlinear Analysis: Real World
  Applications \textbf{10} (2009), no.~5, 2905--2908.

\bibitem{CDesS2018}
F.~Conforto, L.~Desvillettes, and C.~Soresina, \emph{About reaction--diffusion
  systems involving the {H}olling-type {II} and the {B}eddington
  --{D}e{A}ngelis functional responses for predator--prey models}, Nonlinear
  Differential Equations and Applications NoDEA \textbf{25} (2018), no.~3, 24.

\bibitem{DeAngelis1975}
D.~L. DeAngelis, R.~A. Goldstein, and R.~V. O'Neill, \emph{A model for tropic
  interaction}, Ecology \textbf{56} (1975), no.~4, 881--892.

\bibitem{LD_milano}
L.~Desvillettes, \emph{About entropy methods for reaction-diffusion equations},
  Rivista di Matematica dell'Università di Parma \textbf{7} (2007), 81--123.

\bibitem{LD_soresina}
L.~Desvillettes and C.~Soresina, \emph{Non triangular cross-diffusion systems
  with predator-prey reaction terms}, Ricerche di Matematica (2019), 295--314.

\bibitem{FDV2020}
M.~Falconi, Y.~Vera-Dami\'an, and C.~Vidal, \emph{Predator interference in a
  {L}eslie--{G}ower intraguild predation model}, Nonlinear Analysis: Real World
  Applications \textbf{51} (2020), 102974.

\bibitem{GeGyll2012}
S.~Geritz and M.~Gyllenberg, \emph{A mechanistic derivation of the
  {D}e{A}ngelis--{B}eddington functional response}, Journal of theoretical
  biology \textbf{314} (2012), 106--108.

\bibitem{Holling1966}
C.~S. Holling, \emph{The functional response of invertebrate predators to prey
  density}, The Memoirs of the Entomological Society of Canada \textbf{98}
  (1966), no.~S48, 5--86.

\bibitem{HuiBoe1997}
G.~Huisman and R.~J. De~Boer, \emph{A formal derivation of the
  “{B}eddington” functional response}, Journal of theoretical biology
  \textbf{185} (1997), no.~3, 389--400.

\bibitem{Ivlev1961}
V.~S. Ivlev, \emph{Experimental ecology of the feeding of fishes}, (No Title)
  (1961).

\bibitem{Les1}
P.~H. Leslie, \emph{Some {F}urther {N}otes on the {U}se of {M}atrices in
  {P}opulation {M}athematics}, Biometrika \textbf{35} (1948), no.~3/4,
  213--245.

\bibitem{Les2}
\bysame, \emph{A {S}tochastic {M}odel for {S}tudying the {P}roperties of
  {C}ertain {B}iological {S}ystems by {N}umerical {M}ethods}, Biometrika
  \textbf{45} (1958), no.~1/2, 16--31.

\bibitem{Les3}
P.~H. Leslie and J.~C. Gower, \emph{The {P}roperties of a {S}tochastic {M}odel
  for two {C}ompeting {S}pecies}, Biometrika \textbf{45} (1958), no.~3/4,
  316--330.

\bibitem{Merkin1997}
D.~R Merkin, \emph{Introduction to the {T}heory of {S}tability}, vol.~24,
  Springer Science \& Business Media, 1997.

\bibitem{MeDic1986}
J.A. Metz and O.~Diekmann, \emph{The dynamics of physiologically structured
  populations}, vol.~68, Springer, 1986.

\bibitem{Murray1_2002}
J.~D. Murray, \emph{Mathematical biology {I}: {A}n introduction}, vol.~3,
  Springer, 2002.

\bibitem{Murray2_2002}
\bysame, \emph{Mathematical biology {II}: {S}patial models and biomedical
  applications}, vol.~3, Springer New York, 2002.

\bibitem{Yu2014}
S.~Yu, \emph{Global stability of a modified {L}eslie-{G}ower model with
  {B}eddington-{D}e{A}ngelis functional response}, Advances in Difference
  Equations \textbf{2014} (2014), 1--14.

\end{thebibliography}

\end{document}